\documentclass[a4paper,12pt]{amsart}
\usepackage[utf8]{inputenc}

\usepackage{url}
\usepackage[margin=1in]{geometry}  
\usepackage{epsfig}
\usepackage{color}
\usepackage{graphicx}              
\usepackage{amsmath}
\usepackage{amscd}               
\usepackage{amsfonts}
\usepackage{amssymb}             
\usepackage{amsthm}                
\usepackage{epsfig}
\usepackage{mathrsfs}
\usepackage{hyperref}
\usepackage{enumerate}
\usepackage{esint}

\newtheorem{theorem}{Theorem}

\newtheorem*{theorem*}{Theorem}
\newtheorem{lemma}[theorem]{Lemma}
\newtheorem*{lemma*}{Lemma}
\newtheorem{proposition}[theorem]{Proposition}
\newtheorem*{proposition*}{Proposition}
\theoremstyle{definition}

\theoremstyle{remark}

\newcommand{\id}{\operatorname{id}}

\newcommand{\SP}{\mathrm{SP}}
\newcommand{\LL}{\mathrm{L}}
\newcommand{\PL}{\mathrm{PL}}
\newcommand{\EE}{\mathrm{E}}

\newcommand{\tr}{\operatorname{tr}}
\newcommand{\Disc}{\operatorname{Disc}}
\newcommand{\RE}{\operatorname{Re}}
\newcommand{\IM}{\operatorname{Im}}

\newcommand{\E}{\mathbf{E}}     

\newcommand{\Stab}{\operatorname{Stab}}

\newcommand{\bC}{\mathbb{C}}

\newcommand{\bR}{\mathbb{R}}

\newcommand{\zed}{\mathbb{Z}}

\newcommand{\GL}{\mathrm{GL}}

\newcommand{\SL}{\mathrm{SL}}
\newcommand{\SO}{\mathrm{SO}}

\newcommand{\sF}{{\mathscr{F}}}

\newcommand{\sL}{{\mathscr{L}}}

\newcommand{\sS}{{\mathscr{S}}}

\title{Subconvexity of twisted Shintani zeta functions}
\author{Robert D. Hough}
\address{Department of Mathematics, Stony Brook University, 100 Nicolls Road, 
Stony Brook, NY 11794}
\email{robert.hough@stonybrook.edu}
\author{Eun Hye Lee}
\address{Department of Mathematics, Stony Brook University, 100 Nicolls Road, 
Stony Brook, NY 11794}
\email{eunhye.lee@stonybrook.edu }
\begin{document}
 
 \begin{abstract}
 Previously the authors proved subconvexity of Shintani's zeta function enumerating class numbers of binary cubic forms.  Here we return to prove subconvexity of the Maass form twisted version.
 \end{abstract}

 \thanks{This material is based upon work supported by the National Science
Foundation under agreement DMS-1802336. Any opinions, findings and
conclusions or recommendations expressed in this material are those of the
authors and do not necessarily reflect the views of the National Science
Foundation.}

\thanks{Robert Hough is supported by an Alfred P. Sloan Foundation Research 
Fellowship and a Stony Brook Trustees Faculty Award}
\maketitle

\section{Introduction}
In a previous paper \cite{HL22}, the authors proved the subconvexity of Shintani's zeta function enumerating class numbers of binary cubic forms.  The purpose of this article is to demonstrate that the same approach proves subconvexity for the Maass cusp form twisted version of the zeta function \cite{H17}. This is of interest because F. Sato \cite{S94} has demonstrated a general method of twisting zeta functions of prehomogeneous vector spaces with automorphic forms, creating a large class of interesting arithmetic objects, some of which can be proved to be subconvex in this way. 

Let $V_\zed = \{f(x,y) = ax^3 + bx^2 y + cxy^2 + dy^3: a,b,c,d \in \zed\}$ be the space of integral binary cubic forms and let  $\phi$ be a Maass cusp form for $\SL_2(\zed)\backslash \SL_2(\bR)$.  The cusp form twisted zeta function is
\begin{equation}
\sL^{\pm}(\phi,s) := \sum_{\substack{f \in \SL_2(\zed)\backslash V_\zed\\ \pm \Disc(f) >0}} \frac{\phi(f)}{|\Stab(f)|} \frac{1}{|\Disc(f)|^s}, \qquad \RE(s)>1.
\end{equation}
Although this zeta function does not satisfy a known functional equation, its analytic growth is degree four as for the untwisted zeta function and it has holomorphic continuation to $\bC$.  Parallel to the subconvex estimate for the Shintani zeta function, we prove the following subconvex estimate on the critical line for the twisted version.
\begin{theorem}\label{twisted_theorem}
 The twisted Shintani zeta function satisfies the subconvex bound, for any $\epsilon>0$,
 \begin{equation}
  \sL^{\pm}\left(\frac{1}{2} + i\tau, \phi\right) \ll_{\epsilon, \phi} \tau^{\frac{26}{27} +\epsilon}
 \end{equation}
as $\tau \to \infty$.
\end{theorem}
This can be compared to an error term of $\tau^{\frac{98}{99} +\epsilon}$ for the untwisted zeta function, proved in \cite{HL22}.
\subsection{Discussion of method}
Since the twisted version of the Shintani zeta function does not satisfy a known functional equation, a significant part of extending the subconvexity estimate in the untwisted case to the twisted case lies in developing an analogue of the approximate functional equation for this zeta function, developed directly from the theory of prehomogeneous vector space zeta functions.  Once this is developed, methods in the same spirit as those used on the original zeta function suffice to prove the subconvexity.  These include translation between Cartesian and homogeneous coordinates, and van der Corput's method of exponential sums to obtain cancellation in exponential sums.

\section*{Notation}
We use the shorthand $e(x) = e^{2\pi ix}$, $c(x) = \cos(2\pi x)$, $s(x) = \sin(2\pi x)$. On $\bR/\zed$ we use the distance 
\begin{equation}
 \|x\|_{\bR/\zed} = \min_{n \in \zed} |x-n|.
\end{equation}
In $V_{\bR} = \{f(x,y)=ax^3 + bx^2y + cxy^2 + dy^3: a,b,c,d \in \bR\}$ define the infinity ball at $f$ of radius $R$ to be 
\begin{align}
B_R(f) = \{ &a'x^3 + b' x^2y + c' xy^2 + d'y^3: a',b',c',d' \in \bR,\\\notag& \max (|a-a'|,|b-b'|,|c-c'|,|d-d'|) \leq R\}.
\end{align}
The $K$ Bessel function is $K_\nu(x) = \int_0^\infty e^{-x \cosh t}\cosh(\nu t) dt$. The averaging operator $\E_{r \in R} f(r)$ indicates $\frac{1}{|R|} \sum_{r \in R} f(r)$.  We use the following asymptotic notation.  For positive quantities $A, B$ which may depend on the parameter $\tau$, $A = O(B)$ means there is a constant $C>0$ such that $A \leq CB$. This has the same meaning as $A \ll B$.  We write $A \asymp B$ if $A \ll B \ll A$ and $A = o(B)$ if $\lim_{\tau \to \infty} \frac{A}{B} = 0.$ For a function $f$ on $\bR^+$, the Mellin transform is $\tilde{f}(s) = \int_0^\infty f(x)x^{s-1} dx$.  For differential operators with multi-indices $\alpha$, $D_\alpha = \partial_{x_1}^{\alpha_1}...\partial_{x_k}^{\alpha_k}$, $|\alpha| = \alpha_1 + ... + \alpha_k$.  We use the $C^j$ norms on $\bR^n$,
\begin{equation}
 \|f\|_{C^j}  = \sum_{|\alpha|\leq j}   \sup_{x \in \bR^n}\|D^\alpha f(x)\|.
\end{equation}

Our arguments use a smooth partition of unity on the positive reals.  Let $\sigma \geq 0$ be smooth and supported in $\left[\frac{1}{2}, 2\right]$ and satisfy $\sum_{n \in \zed} \sigma(2^n x) \equiv 1$ for $x \in \bR^+$.  

\section{Background}
We prove the following variant of the approximate functional equation in the twisted case in Section \ref{twisted_afe_section}.

\begin{theorem}[Twisted approximate functional equation]\label{twisted_afe_theorem}
 Let $\phi$ be a Hecke-eigen Maass cusp form.  The $\phi$-twisted Shintani $\sL$ function $\sL^{\pm}(s,\phi)$ has the representation at $s = \frac{1}{2}+i\tau$, $\tau \geq 1$,
 \begin{align}
  &\sL^{\pm}\left(\frac{1}{2}+i\tau,\phi\right) \\ \notag &= \sum_{x \in \Gamma \backslash V_\zed \cap V_{\pm}} \frac{V(g_x, |\Disc(x)|)}{|\Stab(x)||\Disc(x)|^{\frac{1}{2}+i\tau}} + \sum_{\xi \in \Gamma \backslash (\hat{V}_\zed \setminus S)} \frac{\hat{V}(g_\xi, |\Disc(\xi)|)}{|\Stab(\xi)||\Disc(\xi)|^{\frac{1}{2}-i\tau}}  + O_A(\tau^{-A}),
 \end{align}
where $g_x, g_\xi$ are in a fundamental domain for $\SL_2(\zed)\backslash \SL_2(\bR)$.  The functions $V$ and $\hat{V}$ satisfy the following properties.

\begin{enumerate}
 \item There is a $c>0$ such that, if $g = n_u a_t k_\theta$ then, for any $\epsilon>0$, \begin{equation}V(g,y), \hat{V}(g,y) \ll \exp(-ct^c) + O_\epsilon\left(y^{-2} \tau^{-\frac{3}{2} + \epsilon}\right).\end{equation}
 \item For any fixed $\eta > 0$, if $h = n_{u'}a_{t'}k_{\theta'}$ with $|u'|, |1-t'|, |\theta'|<\delta < \tau^{-\eta}$ and $t \ll (\log \tau)^{O(1)}$ then, for all $\epsilon >0$ \begin{equation}V(hg,y)-V(g,y), \hat{V}(hg,y)-\hat{V}(g,y) \ll_{\eta, \epsilon} \delta (1 + t)^{O(1)}+y^{-2}\tau^{-\frac{3}{2}+\epsilon}.\end{equation}
 \item   For $a \geq 0$,
 \begin{equation}
  y^a \left(\frac{\partial}{\partial y} \right)^a V(g,y), y^a \left(\frac{\partial}{\partial y}\right)^a \hat{V}(g,y) \ll_{a,A} \left(1 + \frac{y}{1+|\tau|^2}\right)^{-A} + y^{-2}\tau^{-A}.
 \end{equation}
 \end{enumerate}
\end{theorem}

By forming linear combinations to prove Theorem \ref{twisted_theorem} it suffices to prove the following estimates.
\begin{proposition}\label{twisted_afe_proposition}
We have the pair of estimates
\begin{align}
\sum_{x \in \Gamma \backslash V_\zed \cap V_{\pm}} \frac{V(g_x, |\Disc(x)|)}{|\Stab(x)||\Disc(x)|^{\frac{1}{2}+i\tau}}&\ll \tau^{\frac{26}{27}+\epsilon} \\\notag \sum_{\xi \in \Gamma \backslash (\hat{V}_\zed \setminus S)} \frac{\hat{V}(g_\xi, |\Disc(\xi)|)}{|\Stab(\xi)||\Disc(\xi)|^{\frac{1}{2}-i\tau}} &\ll \tau^{\frac{26}{27}+\epsilon}.
\end{align}
\end{proposition}

\subsection{Background regarding $\GL_2(\bR)$}

We use the conventions of \cite{BST13} regarding Lie groups.  Let $\Gamma = \SL_2(\zed)$, and
\begin{align}
G^+ &= \{g \in \GL_2(\bR): \det(g)>0\}\\ \notag K &= \SO_2(\bR), \qquad k_\theta = \begin{pmatrix} c(\theta) & s(\theta)\\ -s(\theta) & c(\theta) \end{pmatrix}\\ \notag
 A_+ &= \left\{ a_t: t \in \bR_+\right\}, \qquad a_t = \begin{pmatrix} \frac{1}{t} &0\\ 0 & t \end{pmatrix}\\ \notag
 N &= \left\{n_u: u \in \bR\right\}, \qquad n_u = \begin{pmatrix} 1 &0\\u&1 \end{pmatrix}\\ \notag
 \Lambda &= \left\{d_\lambda: \lambda \in \bR_+\right\}, \qquad d_\lambda = \begin{pmatrix} \lambda &0\\0&\lambda\end{pmatrix}.
\end{align}

The Iwasawa decomposition of $\SL_2(\bR)$ expresses $g = n_u a_t k_\theta$.    Then for $g \in G^+$,$g = n_u a_t k_\theta d_\lambda$ and Haar measure on $G^+$ is given by $dg = du \frac{dt}{t^3} d\theta \frac{d\lambda}{\lambda}$.  There is an involution $\iota:G^+ \to G^+$, 
\begin{equation}
 g^\iota = \begin{pmatrix} 0&1\\-1&0\end{pmatrix} (g^{-1})^t \begin{pmatrix}0&-1\\1&0\end{pmatrix}.
\end{equation}  In coordinates, $g^\iota = n_u a_t k_\theta d_{\frac{1}{\lambda}}$. 
Let $\sF$ denote the standard fundamental domain for $\SL_2(\zed)\backslash \SL_2(\bR)$,
\begin{align}
 \sF &= \{n_ua_tk_\theta: n_u \in N'(a), a_t\in A', k_\theta \in K\}\\ \notag
 A' &= \left\{ \begin{pmatrix} \frac{1}{t} &0\\0 & t \end{pmatrix}: t \geq \frac{3^{\frac{1}{4}}}{\sqrt{2}}\right\}\\ \notag
 N'(a) &= \left\{ \begin{pmatrix} 1 &0\\u & 1\end{pmatrix}: u\in \nu(a)\right\}
\end{align}
where $\nu(a)$ is the union of two subintervals of $\left[-\frac{1}{2}, \frac{1}{2}\right]$ and is the whole interval if $a \geq 1$.  For constants $A, B>0$, the \emph{Siegel set} $\sS(A,B)$ is
\begin{equation}
 \sS(A,B) = \left\{n_ua_tk_\theta: |u| \leq A, t \geq B, k_\theta \in K\right\}.
\end{equation}

We assume that $F$ is a smooth function, right $K$ invariant, supported on a Siegel set $\sS(A,B)$ with bounded derivatives and such that $\sum_{\gamma \in \Gamma} F(\gamma g) = 1$.  Extend $F$ to $G^+$ by $F(d_\lambda g) = F(g)$ for all $\lambda \in \bR^+$.

\subsection{Automorphic forms}
We consider Maass form $\phi$ which is a smooth function on $\Gamma \backslash \SL_2(\bR)/\SO_2(\bR)$ that is an eigenfunction of the Casimir operator (hyperbolic Laplacian) and Hecke algebra.  The Maass form $\phi$ has a Fourier expansion
\begin{equation}
 \phi\left(\begin{pmatrix} 1&u\\0&1\end{pmatrix} \begin{pmatrix} t &0\\0 & \frac{1}{t}\end{pmatrix} \right) = 2t \sum_{n=1}^\infty \rho_\phi(n) K_{\nu}(2\pi nt^2) \cos(2\pi nu).
\end{equation}
This satisfies $\phi(g) = \phi((g^{-1})^t) = \phi(g^\iota)$.    The Fourier coefficients satisfy the Hecke relations 
\begin{equation}
 \rho_\phi(m)\rho_\phi(n) = \sum_{d|(m,n)} \rho_\phi\left(\frac{mn}{d^2}\right).
\end{equation}
The bound $|\rho_\phi(n)| \ll n^{\frac{7}{64}+\epsilon}$ is proved in \cite{K03}, while $\sum_{n \leq X} |\rho_\phi(n)|^2 \ll X$ follows from Rankin-Selberg Theory.

The $K$-Bessel function satisfies exponential decay in large argument, $K_\nu(z) \ll e^{-z}$ as $z \to \infty$, so that $\phi$ decays rapidly in large $t$. Several integral identities involving the $K$-Bessel function are used,
\begin{align}
 K_\nu(z) &= \frac{1}{2}\left(\frac{z}{2}\right)^\nu \int_0^\infty \exp\left( -t -\frac{z^2}{4t}\right) \frac{dt}{t^{\nu+1}}\\
 \notag K_\nu(xz) &= \frac{\Gamma(\nu+1)(2z)^\nu}{\pi^{\frac{1}{2}} x^\nu} \int_0^\infty \frac{\cos(xt)}{(t^2 + z^2)^{\nu +\frac{1}{2}}} dt\\
 \notag K_\nu(z)K_\nu(\zeta)&= \frac{1}{2} \int_0^\infty \exp\left(-\frac{t}{2} - \frac{z^2+\zeta^2}{2t}\right) K_\nu\left(\frac{z\zeta}{t}\right) \frac{dt}{t}.
\end{align}

The Mellin transform of a pair of $K$ Bessel functions has the following evaluation.
\begin{lemma}\label{double_bessel_mellin_eval}
 We have in $\RE(\alpha + \beta)>0$ and $\RE(s) > |\RE(\mu)| + |\RE(\nu)|$,
 \begin{align}
 \int_0^\infty K_\mu(\alpha x)K_\nu(\beta x) x^{s-1}dx &= 2^{s-3}\alpha^{-s-\nu}\beta^\nu \Gamma\left(\frac{1}{2}(s+\mu+\nu) \right)\Gamma\left(\frac{1}{2}(s-\mu-\nu)\right)\\\notag&\int_0^1 t^{\frac{s-\mu+\nu}{2}-1}(1-t)^{\frac{s+\mu-\nu}{2}-1}\left(1 - \left(1-\frac{\beta^2}{\alpha^2}\right)t\right)^{-\frac{s+\mu+\nu}{2}}dt.
\end{align}
\end{lemma}
\begin{proof}
 The claimed Mellin transform can be evaluated in terms of Gauss's Hypergeometric function,
\begin{equation}
 {}_2F_1(a,b,c;z)= \frac{\Gamma(c)}{\Gamma(b)\Gamma(c-b)}\int_0^1 t^{b-1}(1-t)^{c-b-1}(1-tz)^{-a} dt, \qquad |z|< 1.
\end{equation}
We have in $\RE(\alpha + \beta)>0$ and $\RE(s) > |\RE(\mu)| + |\RE(\nu)|$, (\cite{B54} p.334 eqn (47))
\begin{align}
 \int_0^\infty &K_{\mu}(\alpha x) K_\nu(\beta x) x^{s-1} dx = \frac{2^{s-3}\alpha^{-s-\nu}\beta^\nu}{\Gamma(s)} \Gamma\left(\frac{1}{2}(s+\mu+\nu) \right)\Gamma\left(\frac{1}{2}(s-\mu+\nu) \right) \\&\times \notag \Gamma\left(\frac{1}{2}(s+\mu-\nu) \right) \Gamma\left(\frac{1}{2}(s-\mu-\nu) \right) {}_2F_1\left(\frac{s+\mu+\nu}{2},\frac{s-\mu+\nu}{2},s;1 - \frac{\beta^2}{\alpha^2}\right).
\end{align}
Combining the evaluations proves the lemma. 
\end{proof}

\begin{lemma}\label{double_bessel_mellin_lemma}
 If $\alpha>\beta > 0$ and for some $\epsilon>0$, $\RE(s) > |\RE(\mu)| + |\RE(\nu)| +\epsilon$, then
 \begin{align}
  \int_0^\infty K_\mu(\alpha x)K_\nu(\beta x)x^{s-1}dx &\ll_\epsilon 2^{\RE(s)-3} \alpha^{\RE(\mu)}\beta^{-\RE(s+\mu)}\\ \notag &\times\left|\Gamma\left(\frac{s+\mu+\nu}{2}\right)\Gamma\left(\frac{s-\mu-\nu}{2}\right)\right|.
 \end{align}

\end{lemma}
\begin{proof}
 Bound 
 \begin{equation}
  \left(1 - \left(1 - \frac{\beta^2}{\alpha^2}\right)t\right)^{-\frac{s+\mu+\nu}{2}} \leq \frac{\alpha^{\RE(s+\mu+\nu)}}{\beta^{\RE(s+\mu+\nu)}},
 \end{equation}
 which is the value at $t=1$.
Bound the rest of the integral in $t$ by a constant depending on $\epsilon$.  This gives the claim.
\end{proof}

\begin{lemma}
 The automorphic form is an eigenfunction of the convolution equation
 \begin{equation}
  \int_{G^1} \exp(-\tr g^t g) \phi(hg^{-1}) dg = C\phi(h)
 \end{equation}
with eigenvalue $C = \sqrt{\pi}K_{\nu}(2).$

\end{lemma}
\begin{proof}
The convolution is right $K$ invariant by definition, since $\exp(-\tr g^t g)$ is, and is an eigenfunction of the Casimir operators and Hecke algebra, with the same eigenvalues as $\phi$, by passing the operators under the integral and applying them to $\phi$.  By strong multiplicity 1, the convolution is a multiple of $\phi$.

 To evaluate the eigenvalue, we calculate
\begin{equation}
 \frac{1}{\phi(\id)} \int_{G^1} \phi(g^{-1}) \exp(-\tr g^t g) dg.
\end{equation}
Use $\phi(g^{-1}) = \phi(g^t)$, and integrate away the compact subgroup $K$ under which the  integrand is invariant.  The integral becomes
\begin{align}
 &\frac{1}{\phi(\id)}\int_0^\infty \int_{-\infty}^\infty \exp\left(-t^2 - \frac{1}{t^2} - \left(\frac{u}{t}\right)^2 \right) 2t \sum_{n=1}^\infty \rho_{\phi}(n)K_\nu(2\pi nt^2) \cos(2\pi nu) du \frac{dt}{t^3} \\ \notag
 &= \frac{2 \sum_{n=1}^\infty \rho_\phi(n)}{\phi(\id)}\int_0^\infty \int_{-\infty}^\infty \exp\left(-t^2 -\frac{1}{t^2} -u^2\right) K_\nu(2\pi nt^2) \cos(2\pi nut) du \frac{dt}{t}.
\end{align} 
 Integrating the Gaussian in $u$ obtains
 \begin{align}
 & \sqrt{\pi} \frac{2 \sum_{n=1}^\infty \rho_\phi(n)}{\phi(\id)} \int_0^\infty \exp\left(-t^2(1 + \pi^2 n^2) - \frac{1}{t^2}\right)K_\nu(2\pi nt^2) \frac{dt}{t}\\ \notag
 &= \sqrt{\pi} \frac{2 \sum_{n=1}^\infty \rho_\phi(n)}{\phi(\id)} \frac{1}{2}\int_0^\infty \exp\left(-\frac{t}{2} - \frac{4 + 4\pi^2 n^2}{2t}\right) K_\nu \left(\frac{4\pi n}{t}\right) \frac{dt}{t}\\ \notag
 &= \sqrt{\pi} K_{\nu}(2) \frac{2\sum_{n=1}^\infty \rho_\phi(n)K_{\nu}(2\pi n)}{\phi(\id)} = \sqrt{\pi} K_{\nu}(2).
\end{align}

\end{proof}

\subsection{The prehomogeneous vector space of binary cubic forms}
The space of real binary cubic forms is 
\begin{equation}
 V_{\bR} = \{f(x,y) = ax^3 + bx^2y + cxy^2 + dy^3: a,b,c,d \in \bR\}
\end{equation}
with integral forms $V_\zed$ having $a,b,c,d \in \zed$.  The discriminant is 
\begin{equation}
 \Disc(f) = b^2c^2 - 4ac^3 -4b^3d-27a^2d^2 + 18abcd.
\end{equation}
There is a bilinear pairing on $V_{\bR}$ which identifies it with its dual space,
\begin{equation}
 \langle f, g\rangle = f_1 g_4 - \frac{1}{3} f_2 g_3 + \frac{1}{3} f_3g_2 -f_4g_1.
\end{equation}
The Fourier transform on $V_\bR$ is given by $\hat{F}(\xi) = \int_{V_{\bR}} F(x) e^{-2\pi i \langle x, \xi\rangle} dx$.

The space has a left $\GL_2(\bR)$ action,
\begin{equation}
 \gamma \cdot f(x,y) = \frac{f((x,y)\gamma)}{|\det \gamma|},
\end{equation}
and $\Disc(\gamma \cdot f) = \det(\gamma)^2 \Disc(f)$.
Under this action there are two open orbits $V_{\pm} = \{f: \pm \Disc(f) > 0\}$ and a singular set $S = \{f: \Disc(f) = 0\}$.
The spaces have base points $f_{\pm}$,
\begin{equation}
 f_+ = \frac{1}{(108)^{\frac{1}{4}}} \left(0, 3, 0, -1\right), \qquad f_- = \frac{1}{\sqrt{2}}\left(0,1,0,1\right).
\end{equation}
Both $V_{\pm}$ can be identified as homogeneous spaces for $G^+$.  The mappings 
\begin{align}
 V_+ &= \left\{n_u a_t k_\theta d_\lambda \cdot f_+: u \in \bR, t\in \bR^+, \theta \in \left[0, \frac{1}{3}\right), \lambda \in \bR^+\right\},\\ \notag
 V_- &= \left\{n_u a_t k_\theta d_\lambda \cdot f_-: u \in \bR, t\in \bR^+, \theta \in \left[0, 1\right), \lambda \in \bR^+\right\}
\end{align}
are bijections between $V_{\pm}$ and subsets of $G^+$.  The stabilizer of $f_-$ is trivial and the stabilizer of $f_+$ is the rotation group generated by rotation by $\frac{2\pi}{3}$. 
The bilinear pairing satisfies $\langle x, y \rangle = \langle g\cdot x, g^\iota \cdot y\rangle$.

The rotation $k_\theta$ maps
\begin{align}
 k_\theta \cdot f_- &= \frac{1}{\sqrt{2}}\left(s(\theta), c(\theta), s(\theta), c(\theta) \right)\\ \notag
 k_\theta \cdot f_+ &= \frac{1}{(108)^{\frac{1}{4}}}\left(s(3\theta), 3c(3\theta), -3 s(3\theta),  -c(3\theta)\right).
 \end{align}
Meanwhile $a_t \cdot (a,b,c,d) = \left(\frac{a}{t^3}, \frac{b}{t}, tc, t^3d\right)$ and $n_u\cdot  (a,b,c,d) = (a,  3au+b,  3au^2+2bu + c, au^3 + bu^2 + cu + d)$, $d_\lambda \cdot (a,b,c,d) = (\lambda a, \lambda b, \lambda c, \lambda d)$. 

Putting these formulas together gives the change of coordinates to homogeneous coordinates.
\begin{align}
 n_u a_t k_\theta d_\lambda \cdot f_- &= \frac{\lambda}{\sqrt{2}} \Bigl(t^{-3} s(\theta), 3t^{-3}s(\theta)u + t^{-1}c(\theta), 3t^{-3}s(\theta)u^2 + 2t^{-1}c(\theta)u+ts(\theta), \\&\notag \qquad t^{-3}s(\theta)u^3 + t^{-1}c(\theta)u^2 + ts(\theta)u + t^3 c(\theta) \Bigr)\\
 \notag n_u a_t k_\theta d_\lambda \cdot f_+&= \frac{\lambda}{(108)^{\frac{1}{4}}}\Bigl(t^{-3} s(3\theta), 3t^{-3}s(\theta)u + 3t^{-1}c(3\theta), 3t^{-3}s(\theta)u^2 + 6t^{-1}c(3\theta)u-3ts(3\theta), \\&\notag \qquad t^{-3}s(3\theta)u^3 + 3t^{-1}c(3\theta)u^2 - 3ts(3\theta)u - t^3 c(3\theta) \Bigr).
\end{align}

Given $f \in V_+$, let $n_u a_t k_\theta d_\lambda \cdot f_+ = f$, $u \in \bR, t \in \bR^+, \theta \in \left[0, \frac{1}{3}\right), \lambda \in \bR^+$ and set $g_f = n_u a_t k_\theta d_\lambda$.  Given $f \in V_-$, let $n_u a_t k_\theta d_\lambda \cdot f_- = f$ with $u \in \bR, t \in \bR^+, \theta \in [0, 1), \lambda \in \bR^+$ and set $g_f = n_u a_t k_\theta d_\lambda$.  In these expressions let $u = u(f), t= t(f), \theta = \theta(f)$.  

We include several lemmas from \cite{HL22} needed in the proof.
\begin{lemma}
 Suppose that $f = n_u a_t \cdot f_{\pm}$ then $\log t = \log t_f + O(1)$.
\end{lemma}
\begin{proof}
 This is \cite{HL22} Lemma 5.
\end{proof}

Integrals over $V_\pm$ may be expressed, for $f \in C_0(V_{\pm})$,
\begin{align}
\label{integration_formulae} \int_{V_+} f(v) \frac{dv}{\Disc(v)} &= \int_{-\infty}^\infty \int_{0}^\infty \int_0^{\frac{1}{3}} \int_0^\infty f(n_u a_t k_\theta d_\lambda \cdot f_+) \frac{d\lambda}{\lambda} d\theta \frac{dt}{t^3} du\\ \notag
  \int_{V_-} f(v) \frac{dv}{|\Disc(v)|} &= \int_{-\infty}^\infty \int_{0}^\infty \int_0^{1} \int_0^\infty f(n_u a_t k_\theta d_\lambda \cdot f_-) \frac{d\lambda}{\lambda} d\theta \frac{dt}{t^3} du.
\end{align}
Thus $|da \wedge db \wedge dc \wedge dd| = \frac{\lambda^3}{t^3} |d\lambda \wedge d\theta \wedge dt \wedge du|$.

\begin{lemma}\label{Jacobian_lemma}
 When $u, t, \theta$ vary in a Siegel set and $\lambda \geq 1$,  the change of coordinates $(a,b,c,d)= n_u a_t k_\theta d_\lambda \cdot f_{\pm}$ satisfies
 \begin{align}
  \frac{\partial(a,b,c,d)}{\partial( u,t, \theta, \lambda)} &= \begin{pmatrix}0 & O(\lambda t^{-3})& O(\lambda t^{-1})& O(\lambda t) \\ O(\lambda t^{-4}) & O(\lambda t^{-2})& O(\lambda) & O(\lambda t^2)\\ O(\lambda t^{-3}) & O(\lambda t^{-1})& O(\lambda t)& O(\lambda t^3)\\
  O(t^{-3}) & O(t^{-1}) & O(t) & O(t^3)\end{pmatrix}\\ \notag
  \frac{\partial( u,t, \theta, \lambda)}{\partial(a,b,c,d)} &= \begin{pmatrix}O(\lambda^{-1}t^5) & O(\lambda^{-1}t^4)& O(\lambda^{-1}t^3)& O(t^3)\\ O(\lambda^{-1}t^3)&O(\lambda^{-1}t^2)&O(\lambda^{-1}t)& O(t)\\
  O(\lambda^{-1}t) & O(\lambda^{-1})& O(\lambda^{-1}t^{-1})& O(t^{-1}) 
  \\
  O(\lambda^{-1}t^{-1})& O(\lambda^{-1}t^{-2})& O(\lambda^{-1}t^{-3}) & O(t^{-3})\end{pmatrix}.
 \end{align}

\end{lemma}
\begin{proof}
This is \cite{HL22} Lemma 6.
\end{proof}

The proof of the theorems rely on estimates for the derivatives and logarithmic derivatives of the discriminant which are uniform in the cuspidal parameter $t$.

For a multi-index $\alpha = (\alpha_1, \alpha_2, \alpha_3, \alpha_4)$ let $|\alpha| = \alpha_1 + \alpha_2 + \alpha_3 + \alpha_4$ and let 
$D^\alpha = D_a^{\alpha_1}D_b^{\alpha_2}D_c^{\alpha_3}D_d^{\alpha_4}$.
\begin{lemma}\label{derivative_lemma}
  Let $f= n_u a_t d_\lambda \cdot f_\pm$ with $u = O(1)$ and $t, \lambda \gg 1$. Then
\begin{equation}
 D^\alpha \Disc(f) = O\left(\lambda^{4-|\alpha|}t^{3|\alpha|}\right)
\end{equation}
and
\begin{equation}
 D^\alpha \log |\Disc(f)| = O\left(\lambda^{-|\alpha|}t^{3|\alpha|}\right)
\end{equation}
while
\begin{equation}
 \max_{D \in \{D_a^3, D_b^3, D_c^3, D_d^3\}} |D\log |\Disc(f)|| \gg \frac{1}{t^9 \lambda^3}.
\end{equation}

\end{lemma}

\begin{proof}
This is \cite{HL22} Lemma 7.
\end{proof}

\begin{lemma}\label{f_variation_lemma}
 Let $f = (a,b,c,d) \in \bR^4$ be a form with $\lambda_f \geq 1$ and $t_f \gg 1$.  For every constant $C_1 > 1$ there is a constant $C_2 > 0$ so that if $\|\tilde{f} - f\|_2 \leq C_2 \frac{\lambda_f}{t_f^3}$ then 
 \begin{equation}
  \frac{\lambda_f}{C_1} \leq \lambda_{\tilde{f}} \leq C_1 \lambda_f.
 \end{equation}
 For every constant $C_3 > 1$ there is a constant $C_4 > 0$ so that if $\|\tilde{f} - f\|_2 \leq C_4 \frac{\lambda_f}{t_f^3}$ then 
 \begin{equation}
  \frac{t_f}{C_3} \leq t_{\tilde{f}} \leq C_3 t_f.
 \end{equation}
 
\end{lemma}
\begin{proof}
This is \cite{HL22} Lemmas 8 and 9.
\end{proof}

\begin{lemma}\label{exponential_sum_lemma}
Let $Y\gg \tau^{\frac{4}{3}}T_1^{12}$ and $T_1 = o\left(Y^{\frac{1}{84}} \right)$.  With the choice $R = \frac{Y^{\frac{1}{4}}}{\tau^{\frac{7}{27}} T_1^{\frac{7}{3}}}$, we have the bound
 \begin{equation}
  \E_{y \in B_{R}}\left[|\Disc(f+y)|^{-i\tau}\right] \ll   \frac{T_1^{\frac{8}{3}}}{\tau^{\frac{1}{27}}}(\log \tau)^{\frac{4}{9}}.
 \end{equation}

\end{lemma}

\begin{proof}
 This is \cite{HL22} Lemma 13.
\end{proof}


\section{Proof of Theorem \ref{twisted_theorem}}

We give the Proof of Theorem \ref{twisted_theorem} assuming the approximate functional equation, which is proved in the remainder of the paper using oscillatory integrals.

\begin{proof}[Proof of Theorem \ref{twisted_theorem}]
Let $F$ be a smooth function supported on a Siegel set, $K$ invariant with $\partial_t F, \partial_u F = O(1)$, such that $\sum_{\gamma \in \Gamma} F(\gamma \cdot g) = 1$.

By the variant of the approximate functional equation
 \begin{align}
  \sL^{\pm}(s,\phi)&= \sum_{x \in V_\zed \cap V_{\pm}} \frac{F(g_x)V(g_x, |\Disc(x)|)}{|\Stab(x)||\Disc(x)|^s}+ \sum_{\xi \in \hat{V}_\zed \setminus S} \frac{F(g_\xi)\hat{V}(g_\xi, |\Disc(\xi)|)}{|\Stab(\xi)||\Disc(\xi)|^{1-s}} + O_A(\tau^{-A}).
 \end{align}
The proof in case of forms and dual forms is the same, so we just show the proof for dual forms. By \cite{S72} p.169 Corollary, the number of dual forms of discriminant $m>0$ with $|\Stab(\xi)|>3$ is
\begin{equation}
 2\hat{h}_2(m) = \#\{(x,y) \in \zed^2: 81(x^2 + xy + y^2)^2 = m\}.
\end{equation}
Since the quadratic form is definite, summing over $x, y$ proves that the contribution of these forms is \begin{equation}\ll \sum_{(x,y) \neq (0,0)} \frac{1}{x^2 + xy +y^2} \left(1 + \frac{x^2 + xy + y^2}{\tau^2}\right)^{-A} \ll \log \tau.\end{equation}  Thus the factor of the stabilizer is dropped in what follows.

Write the dual sum as
\begin{equation}
 \sum_{\xi \in \hat{V}_\zed \setminus S} \sum_{n \geq A} \sigma\left(\frac{|\Disc(\xi)|}{2^n} \right)F(g_\xi)\frac{\hat{V}_{\pm}(g_\xi, |\Disc(\xi)|)}{|\Disc(\xi)|^{1-s}}.
\end{equation}
 The sum of interest is equal to the averaged sum 
\begin{equation}
 \Sigma = \sum_{\xi \in \hat{V}_\zed } \E_{y \in B_{R(Y)}} \sigma\left(\frac{|\Disc(\xi+y)|}{Y} \right)F(g_{\xi+y})\frac{\hat{V}_{\pm}(g_{\xi+y}, |\Disc(\xi+y)|)}{|\Disc(\xi+y)|^{1-s}}.
\end{equation}
 For those  $\xi$ for which $t \gg (\log \tau)^{O(1)}$, use item (1) of Theorem \ref{twisted_afe_theorem} to bound \begin{equation}\hat{V}_{\pm}(g_\xi, |\Disc(\xi)|) \ll \min\left(\exp(-c t^c)+O_A(y^{-2}\tau^{-A}), \left(1 + \frac{|\Disc(\xi)|}{\tau^2}\right)^{-A}\right).\end{equation}  Thus these can be discarded with error $O_A(\tau^{-A})$. 

Use that, for $v \in \bR^4$, $\|v\|=1$,
\begin{equation}
 \partial_v \lambda \ll t^3,\; \partial_v t \ll \frac{t^4}{\lambda},\; \partial_v u \ll \frac{t^5}{\lambda},\; \partial_v \theta \ll \frac{t^3}{\lambda}.
\end{equation}
Thus $|\Disc(\xi+y)| = |\Disc(\xi)| + O\left(t^3 Y^{\frac{3}{4}}R\right)$, and if $g_\xi = n_u a_t k_{\theta}$, $g_{\xi+y} = n_{u'} a_{t'}k_{\theta'}$, $g_{\xi+y}g_\xi^{-1} = n_{u_1}a_{t_1}k_{\theta_1}$ then 
\begin{equation}
 |u-u'| \ll \frac{t^5 R}{Y^{\frac{1}{4}}}, \qquad |t-t'| \ll \frac{t^4 R}{Y^{\frac{1}{4}}}, \qquad |\theta-\theta'| \ll \frac{t^3 R}{Y^{\frac{1}{4}}}
\end{equation}
and $|u_1|, |t_1|, |\theta_1| \ll \frac{t^{O(1)}R}{Y^{\frac{1}{4}}}$.  Applying item (2) of Theorem \ref{twisted_afe_theorem} it follows that
\begin{align}
 \Sigma &= \sum_{\xi \in \hat{V}_\zed, t\ll (\log t)^{O(1)}} \sigma\left( \frac{|\Disc(\xi)|}{Y}\right)F(g_\xi) \frac{\hat{V}_{\pm}(g_\xi, |\Disc(\xi)|)}{\sqrt{|\Disc(\xi)|}} \E_{y \in B_{R(Y)}}[|\Disc(\xi+y)|^{i\tau} ]\\ \notag&+ O \left(Y^{\frac{1}{4}}R\tau^\epsilon \left(1 + \frac{Y}{\tau^2}\right)^{-A}\right).
\end{align}
 Lemma \ref{exponential_sum_lemma} now applies as before and estimates the expectation over $y$ by $O\left(\frac{1}{\tau^{\frac{1}{27}-\epsilon}} \right)$ with a choice of $R$ of $\frac{Y^{\frac{1}{4}}}{\tau^{\frac{7}{27}-\epsilon}}$.  This completes the proof.

\end{proof}

\section{Estimates for oscillatory integrals}
This section collects together estimates for oscillatory integrals used in proving Theorem \ref{twisted_afe_theorem}.  Stein \cite{S93} gives a thorough treatment of oscillatory integrals as used in this section. In one variable, the standard type integral is 
\begin{equation}
 I(\lambda) = \int_a^b e^{i \lambda \phi(x)} \psi(x) dx
\end{equation}
where $\phi, \psi$ are smooth, $\phi'(x) \neq 0$ and $\psi$ is compactly supported in $[a,b]$.  Define the differential operator $Df(x) = \frac{1}{i\lambda \phi'(x)} \frac{df}{dx}$, with transpose operator $D^t f = -\frac{d}{dx}\left(\frac{f}{i\lambda \phi'(x)} \right)$.  Then $D^N e^{i\lambda \phi(x)} = e^{i\lambda \phi(x)}$, and hence, repeated integration by parts obtains, for any $N \geq 0$,
\begin{equation}
 \int_a^b e^{i\lambda \phi(x)}\psi(x)dx = \int_a^b e^{i\lambda \phi(x)} (D^t)^N(\psi)(x) dx.
\end{equation}

The following lemma is useful for giving uniform estimates.
\begin{lemma}
 The operator $(D^t)^N\psi$ is a sum of monomials with coefficients bounded only in terms of $N$,
 \begin{equation}
  (D^t)^N(\psi) = \sum_{\substack{\alpha = (\alpha_0, ..., \alpha_k)\\|\alpha| = N}} C_{N, \alpha} \frac{\partial^{\alpha_0}\psi \partial^{\alpha_1+1}\phi \partial^{\alpha_2+1}\phi ... \partial^{\alpha_k+1}\phi}{\lambda^N(\phi')^{N+k}}.
 \end{equation}

\end{lemma}
\begin{proof}
 Applying $D^t$ $N$ times differentiates a total of $N$ times, multiplies by a factor of $\lambda$ $N$ times and divides by $\phi'$ $N$ times.  Each derivative goes on either $\psi$, the denominator, or a derivative of $\phi$ in the numerator.  Differentiating the denominator increases the power of the denominator by 1, and introduces a new factor of a derivative of $\phi$ in the numerator.
\end{proof}

In higher dimensions, Stein \cite{S93} proves the following estimates.
\begin{lemma}
 Suppose $\psi$ is smooth, has compact support, and that $\phi$ is a smooth real-valued function that has no critical points in the support of $\psi$.  Then 
 \begin{equation}
  I(\lambda) = \int_{\bR^n} e^{i\lambda \phi(x)} \psi(x)dx = O(\lambda^{-N})
 \end{equation}
as $\lambda \to \infty$ for every $N \geq 0$.

\end{lemma}
\begin{proof}
 This is \cite{S93} Proposition 4 on p.341.
\end{proof}

When there is a critical point $x_0$ of $\phi$ such that $\left[ \frac{\partial^2 \phi}{\partial x_i \partial x_j}\right](x_0)$ is invertible, then $x_0$ is non-degenerate.
\begin{lemma} \label{stationary_phase_asymptotic_lemma}
 Suppose $\phi(x_0) = 0$, and $\phi$ has a nondegenerate critical point at $x_0$. If $\psi$ is supported in a sufficiently small neighborhood of $x_0$, then
 \begin{equation}
  \int_{\bR^n} e^{i\lambda \phi(x)} \psi(x) dx \sim \lambda^{-\frac{n}{2}} \sum_{j=0}^\infty a_j \lambda^{-j}
 \end{equation}
as $\lambda \to \infty$. The coefficients $a_j$ are specified by finitely many derivatives of $\phi$ and $\psi$ at $x_0$. The asymptotic holds in the sense that, for every $r, N \geq 0$,
\begin{equation}
 \left(\frac{d}{d\lambda}\right)^r\left[I(\lambda) - \lambda^{-\frac{n}{2}} \sum_{j=0}^N a_j \lambda^{-j}\right] = O\left(\lambda^{-\frac{n}{2}-r-N-1}\right). 
\end{equation}

\end{lemma}
\begin{proof}
 This is \cite{S93} p.344 Proposition 6.
\end{proof}

In Section \ref{twisted_afe_section} the following oscillatory integral is needed.  Let $g_1 \in G^1$ and $w \in \bC$ and define  
\begin{align}
 &E(g_1, w)=\int_0^\infty\int_{G^1}  \lambda^{4(1-s+w)} \exp\left(- \tr g_2^t (g_1^{-1})^t g_1^{-1}g_2 \right)e\left(-\lambda \langle k_{\theta_2} \cdot f_{\pm}, a_{t_2}^{-1} n_{u_2}^{-1} \cdot \tilde{f}_{\pm}\rangle \right)\frac{d\lambda}{\lambda} dg_2.
\end{align}
Note that this integral is conditionally convergent in $\lambda$.  However, the integral in $G^1$ is of the type described by $I(\lambda)$, and thus outside a bounded set of points of stationary phase, saves an arbitrary power of $\lambda$, hence guaranteeing convergence.  
The contribution of the points of stationary phase are meromorphic in $w$ and hence are equal in the sense of meromorphic continuation to their regularized value.  We choose to regularize by introducing a smooth partition of unity $\sum_{n \in \zed} \sigma\left(\frac{\lambda}{2^n} \right)=1$ where $\sigma$ has compact support on $\bR^+$.  The regularized integral is the sum of the integrals
\begin{equation}
 \int_0^\infty\int_{G^1} \sigma\left(\frac{\lambda}{2^n} \right) \lambda^{4(1-s+w)} \exp\left(- \tr g_2^t (g_1^{-1})^t g_1^{-1}g_2 \right)e\left(-\lambda \langle k_{\theta_2} \cdot f_{\pm}, a_{t_2}^{-1} n_{u_2}^{-1}\cdot \tilde{f}_{\pm}\rangle \right)\frac{d\lambda}{\lambda} dg_2
\end{equation}
which give the correct value in the domain of absolute convergence and have analytic continuation. The evaluation of these points can be obtained by using Lemma \ref{stationary_phase_asymptotic_lemma} to separate an asymptotic main term in falling powers of $\lambda$, together with an error term that is absolutely convergent in $\lambda$.  The main term becomes a Mellin transform of sin and cosine functions, which have meromorphic continuation. 

In coordinates, the trace in the exponential function of $E(g_1,w)$ is equal to 
\begin{equation}
 \left(\frac{t_1}{t_2}\right)^2 + \left(\frac{t_2}{t_1}\right)^2 + \left(\frac{t_2}{t_1}u_2 - \frac{t_1}{t_2}u_1 \right)^2.
\end{equation}
We have
\begin{align}
 k_\theta \cdot f_{-} &= \frac{1}{\sqrt{2}}(s(\theta), c(\theta), s(\theta), c(\theta)),\\ \notag
 k_\theta \cdot f_{+} &= \frac{1}{(108)^{\frac{1}{4}}}(s(3\theta), 3c(3\theta), -3s(3\theta), -c(3\theta)),
\end{align}
and
\begin{align}
 a_{\frac{1}{t}}n_{-u} \cdot f_{-} &= \frac{1}{\sqrt{2}}\left(0, t, \frac{-2u}{t}, \frac{1+u^2}{t^3}\right),\\ \notag
 a_{\frac{1}{t}}n_{-u} \cdot f_{+} &= \frac{1}{(108)^{\frac{1}{4}}}\left( 0, 3t, \frac{-6u}{t}, \frac{-1+3u^2}{t^3}\right).
\end{align}
This gives rise to four pairings,
\begin{equation}
 \begin{tabular}{|l|l|}
 \hline
  $--$& $\frac{1}{2} \left(\frac{1 + \frac{1}{3}t^4 + u^2}{t^3}s(\theta) + \frac{2u}{3t}c(\theta) \right)$\\
  \hline
  $-+$& $\frac{1}{2 \cdot 3^{\frac{3}{4}}}\left(\frac{-1 + t^4 +3u^2}{t^3}s(\theta) + \frac{2u}{t}c(\theta) \right) $\\
  \hline
  $+-$ & $\frac{1}{2\cdot 3^{\frac{3}{4}}} \left(\frac{1 -t^4 + u^2}{t^3}s(3\theta) +\frac{2u}{t}c(3\theta)\right)$\\
  \hline
  $++$ & $\frac{1}{2 \cdot 3^{\frac{3}{2}}}\left(\frac{-1 -3t^4 +3u^2}{t^3}s(3\theta) + \frac{6u}{t}c(3\theta) \right) $\\ \hline
 \end{tabular}
\end{equation}
We show how to analyze the first pairing, and leave the remaining pairings to the reader, which are similar.

Let 
\begin{equation}
 P_{\pm, \pm}(t,u) = 2 \int_{\bR/\zed} \langle k_\theta \cdot f_{\pm}, a_t^{-1}n_u^{-1} \cdot \tilde{f}_{\pm}\rangle^2 d\theta.
\end{equation}
Thus, after a change of coordinates,
\begin{align}
 &E(g_1, w)=\\
 \notag & \int_0^\infty\int_{G^1}  \lambda^{4(1-s+w)} \exp\left(- \tr g_2^t (g_1^{-1})^t g_1^{-1}g_2 \right)e\left(\lambda \sqrt{P_{\pm,\pm}(t,u)}c(\theta)\right) \frac{dt}{t^3} du d\theta \frac{d\lambda}{\lambda}.
\end{align}
Define the phase function
\begin{equation}
 \Phi(\lambda, \theta, t, u) = -4(\tau -\IM(w))\log \lambda + 2\pi \lambda \sqrt{P_{\pm,\pm}\left(t,u \right)}c(\theta)
\end{equation}
and the weight function
\begin{equation}
 W(\lambda, t, u) = \frac{\lambda^{1 + 4 \RE(w)}}{t^3} \exp\left(-\left(\frac{t_1}{t} \right)^2 -\left(\frac{t}{t_1}\right)^2(1 + u^2) \right)
\end{equation}
so that 
\begin{equation}
 E(g_1, w) = \int_{-\infty}^\infty \int_0^1 \int_0^\infty\int_0^\infty e^{i \Phi\left(\lambda, \theta, t, u + \left(\frac{t_1}{t} \right)^2 u_1\right)}W(\lambda, t, u) d\lambda d\theta dt du.
\end{equation}

The following Lemma records information regarding the points of stationary phase of $\Phi$. 
\begin{lemma}\label{oscillatory_integral_derivatives}
 Let $\IM(w) = O\left(\tau^{\frac{1}{2}+\epsilon}\right)$.  The stationary phase equation $D\Phi = 0$ has boundedly many solutions, all of which satisfy $\theta \in \zed$, $u = O(1)$, $\log t = O(1)$ and $\lambda \asymp \tau$.  Let $\lambda' = \frac{\lambda}{\tau - \IM(w)}$.  At a point of stationary phase
 \begin{equation}
  \frac{1}{\tau - \IM(w)} D^n \Phi(\lambda', \theta, t, u) = O_n(1).
 \end{equation}
The conditions
\begin{equation}
 \partial_\theta \Phi = \partial_\lambda \Phi = \partial_t \Phi = 0
\end{equation}
imply $\theta \in \zed$, $\lambda = \frac{2(\tau - \IM(w))}{\pi \sqrt{P_{\pm,\pm}(t,u)}}$ and $\partial_t P_{\pm,\pm}(t,u) = 0$.  This requires $t \gg 1$ and $u^2 \asymp t^4$.
\end{lemma}
\begin{proof}
 Solving $\partial_\lambda \Phi = \partial_\theta\Phi = 0$ imposes $c(\theta)=1$ and $\lambda = \frac{2 (\tau-\IM(w))}{\pi\sqrt{P_{\pm,\pm}(t,u)}}$.  Treating $\partial_t \Phi$ and $\partial_u \Phi$ we may impose $\partial_t P_{\pm,\pm}(t,u) = 0$ and $\partial_u P_{\pm,\pm}(t,u) = 0$ instead.  If $t \ll 1$ replace $t$ with $s = \frac{1}{t}$ and impose $\partial_s P_{\pm,\pm}(s^{-1},u) = 0$.  For some non-zero constants $\alpha, \beta, \gamma, \delta$ this requires
 \begin{align}
  &\frac{\partial}{\partial s} \left[ s^2 \left(\left(\alpha s^2 + \frac{\beta}{s^2} + \gamma u^2s^2\right)^2 + \delta^2 u^2 \right)\right] \\
  \notag&= 2s \left[\left(\left(\alpha s^2 + \frac{\beta}{s^2} + \gamma u^2s^2\right)^2 + \delta^2 u^2 \right) \right] + 2s^2 \left(\alpha s^2 + \frac{\beta}{s^2} + \gamma u^2s^2\right)\left(2\alpha s - 2\frac{\beta}{s^3} + 2\gamma u^2s\right)\\
  \notag&= 2s \left[\left(\alpha s^2 + \frac{\beta}{s^2} + \gamma u^2s^2\right)^2 +\delta^2 u^2+ 2 (\alpha s^2 + \gamma u^2 s^2)^2 - 2 \frac{\beta^2}{s^4} \right]=0.
 \end{align}
 This has no solution if $s$ is sufficiently large.  Thus assume $t \gg 1$.  
Now we wish to solve 
\begin{equation}
 \partial_t \left[t^2 \left[\left(\frac{\alpha}{t^4} + \beta + \gamma \frac{u^2}{t^4}\right)^2 + \delta^2 \frac{u^2}{t^4}\right]\right] = 0.
\end{equation}
If the term $\frac{\alpha}{t^4}$ is treated as negligible as $t \to \infty$, then this results in a polynomial equation in $\frac{u^2}{t^4}$, which has bounded solutions.  This proves the claims regarding $\partial_t \Phi$.  To prove the final claim, we can differentiate with respect to $u^2$ instead, which obtains a linear equation in $u^2$.  Substituting the resulting value of $u^2$ as a function of $t$ now obtains a polynomial equation in $t$ for $\partial_t \Phi = 0$ which forces $t$ to be bounded.

The claim regarding higher derivatives at points of stationary phase hold because all variables are bounded there.
\end{proof}

Estimates for the partial derivatives of the phase function $\Phi$ and the weight function $W$ are recorded in the following lemmas.
\begin{lemma}\label{phase_fn_lemma} Let $\IM(w) = O\left(\tau^{\frac{1}{2}+\epsilon}\right)$.
 We have the partial derivatives ($c_n, c_n'$ denote constants that depend on $n$),
 \begin{align}
  \partial_\lambda \Phi &= \frac{-4(\tau-\IM(w))}{\lambda} + 2\pi \sqrt{P_{\pm,\pm}(t,u)}c(\theta),\\ \notag
  \partial_\lambda^n \Phi &= \frac{c_n (\tau -\IM(w))}{\lambda^{n+1}}, \qquad n > 1
  \\ \notag
  \partial_\theta^n \Phi &= c_n' \lambda \sqrt{P_{\pm, \pm}(t,u)}\times \left\{\begin{array}{lll} c(\theta) && n \text{ even}\\ s(\theta) && n \text{ odd} \end{array}\right., \qquad n \geq 1.
 \end{align}
 Meanwhile, for $t \gg 1$, $\sqrt{P_{\pm,\pm}(t,u)} \gg t + \frac{u^2}{t^3}$. For $t \ll 1$, $\sqrt{P_{\pm,\pm}(t,u)} \gg \frac{1}{t}$.  Also,
 \begin{align}
  \partial_t^n P_{\pm,\pm}(t,u) &\ll \left\{\begin{array}{lll} \frac{1}{t^n} \left(t^2 + \frac{u^4}{t^6}\right) && t \gg 1 \\ \frac{1}{t^n} \left(\frac{1 + u^4}{t^6}\right)&& t\ll 1 \end{array}\right.\\
  \notag \partial_u^n P_{\pm,\pm}(t,u) &\ll \frac{1}{|u|^n} \left(\frac{u^2}{t^2} + \frac{u^4}{t^6}\right).
 \end{align}
In particular, when $t \gg 1$,
\begin{align}
\partial_t^n \sqrt{P_{\pm,\pm}(t,u)} &\ll_n  \frac{1}{t^n}\left(t + \frac{u^2}{t^3}\right), \\ \notag
\partial_u^n \sqrt{P_{\pm,\pm}(t,u)} &\ll_n \frac{1}{|u|^n} \left(t + \frac{u^2}{t^3}\right).
\end{align}
\end{lemma}
\begin{lemma}\label{weight_function_lemma_W} Let $\IM(w) = O\left(\tau^{\frac{1}{2}+\epsilon}\right)$.
 The partial derivatives of the weight function are bounded for $n \geq 1$ by
 \begin{align}
  \partial_t^n W &\ll_n \frac{W}{t^n} \left(1 + \left(\frac{t_1}{t}\right)^2 + \left(\frac{t}{t_1}\right)^2(1 + u^2) \right)^n,\\
  \notag \partial_\lambda^n W &\ll_n \frac{W}{\lambda^n},\\
  \notag \partial_u^n W & \ll_n W\left[\left(\left(\frac{t}{t_1}\right)^2 + \left(\frac{t_1}{t}\right)^2\right)(1 +|u|)\right]^n .
 \end{align}

\end{lemma}

In order to evaluate the oscillatory integral $E(g_1)$ we decompose it into several smooth pieces, thus identifying the parts that make the dominant contribution.   Let
\begin{equation}
 \sS = \{(\lambda_0, \theta_0, t_0, u_0): D\Phi = 0\}
\end{equation}
be the set of points of stationary phase.
Let
\begin{equation}
 \sS'(u) = \{(\lambda_0, \theta_0, t_0, u): \partial_\lambda \Phi = \partial_t\Phi = \partial_\theta\Phi = 0 \}.
\end{equation}
Let $\psi$ be a smooth function on $\bR$, supported in $\left[-\frac{1}{2}, \frac{1}{2}\right]$, identically 1 in a neighborhood of 0.

Define 
\begin{align}
\Psi_{\SP}(\lambda, \theta, t,u) &= \sum_{(\lambda_0, \theta_0, t_0, u_0) \in \sS}\psi\left(\frac{\lambda - \lambda_0}{\tau^{\frac{1}{2}+\epsilon}} \right)\psi\left(\frac{\theta-\theta_0}{\tau^{-\frac{1}{2}+\epsilon}} \right)\psi\left(\frac{t-t_0}{\tau^{-\frac{1}{2}+\epsilon}} \right)\psi\left(\frac{u-u_0}{\tau^{-\frac{1}{2}+\epsilon}} \right)\\
\notag\Psi_{\LL}(\lambda, \theta, t,u) &= \sum_{(\lambda_0, \theta_0, t_0, u) \in \sS'(u)} \psi\left(\frac{\lambda_0(\lambda-\lambda_0)}{\tau^{\frac{1}{2}+\epsilon}} \right)\psi\left(\frac{\theta-\theta_0 }{\tau^{-\frac{1}{2}+\epsilon}} \right)\psi\left(\frac{\lambda_0(t-t_0)}{\tau^{\frac{1}{2}+\epsilon}} \right)\\\notag&\times\psi\left(\frac{\log t - \frac{1}{2}\log \tau}{\epsilon \log \tau} \right)\left(1-\psi\left(\frac{\tau^{2+\epsilon}}{y P_{\pm,\pm}(t,u)^2} \right)\right)\\\notag
\Psi_{\PL}(\lambda, \theta, t,u) &= \psi\left(\frac{\tau^{2+\epsilon}}{y P_{\pm, \pm}(t,u)^2} \right)
\end{align}
and $\Psi_{\EE}= 1-\Psi_{\SP} - \Psi_{\LL} - \Psi_{\PL}.$  Similarly define $E_{\SP}(g_1), E_{\LL}(g_1), E_{\PL}(g_1), E_{\EE}(g_1)$ by
\begin{align}
 &E_{*}(g_1, w)=\\\notag&\int_{-\infty}^\infty \int_0^1 \int_0^\infty\int_0^\infty \Psi_*\left(\lambda, \theta, t, u + \left(\frac{t_1}{t} \right)^2 u_1\right)e^{i \Phi\left(\lambda, \theta, t, u + \left(\frac{t_1}{t} \right)^2 u_1\right)}W(\lambda, t, u) d\lambda d\theta dt du. 
\end{align}

\begin{lemma}\label{E_est_lemma} Let $\IM(w) = O\left(\tau^{\frac{1}{2}+\epsilon}\right)$.
 We have the bounds
 \begin{align}
  E_{\SP}(g_1, w) &\ll\tau^{4\RE(w)} \sum_{(\lambda_0, \theta_0, t_0, u_0) \in \sS} \exp\left(-\left(\frac{t_1}{t_0} \right)^2-\left(\frac{t_0}{t_1} \right)^2 -\left(\frac{t_0}{t_1}u_0 - \frac{t_1}{t_0}u_1 \right)^2 \right)\\ & + \notag O_A(\tau^{-A})\\
  \notag E_{\EE}(g_1,w) &\ll_A \tau^{-A}.
 \end{align}
If $y \ll \tau^{O(\epsilon)}$, $ t_1 = \tau^{\frac{1}{2} + O(\epsilon)}$ and $\frac{u_1^2}{t_1^4} = \tau^{O(\epsilon)}$ then $E_{\LL}(g_1,w) \ll \tau^{2\RE(w) -\frac{3}{2} + O(\epsilon)}$.  Otherwise, $E_{\LL}(g_1,w) = O_A(\tau^{-A})$.
\end{lemma}
\begin{proof}
To treat $E_{\SP}(g_1,w)$, restrict to the case $\frac{t_1}{t_0}, \frac{t_0}{t_1} \ll \log \tau$, $\left|\frac{t_0}{t_1}u_0 - \frac{t_1}{t_0} u_1\right| \ll \log \tau$, since outside this range, the exponential factor in the weight function $W$ makes the integral $O_A(\tau^{-A})$.  In this case, throughout the domain of integration, $\lambda \asymp \tau$ and \begin{equation}W \ll \tau^{1 + 4 \RE(w)}\exp\left(-\left(\frac{t_1}{t_0} \right)^2-\left(\frac{t_0}{t_1} \right)^2 -\left(\frac{t_0}{t_1}u_0 - \frac{t_1}{t_0}u_1 \right)^2 \right).\end{equation}
Replacing $\lambda' = \frac{\lambda}{\tau - \IM(w)}$, then Taylor expanding the phase and weight function to degree 3 obtains the correct order of magnitude for the integral around each point in $\sS$.

Now consider the case of $E_{\LL}$ and $E_{\EE}$.  Note that both of these integrals are constrained by $yP_{\pm,\pm}^2 \ll \tau^{2+\epsilon}$, which implies $t, \frac{u^2}{t^3} \ll \frac{\tau^{\frac{1}{2} + \frac{\epsilon}{4}}}{y^{\frac{1}{4}}}$ and when $t \ll 1$,
\begin{equation}
 t \gg \frac{1}{y^{\frac{1}{4}}\tau^{\frac{1}{2}-\frac{\epsilon}{4}}},\qquad |u| \ll \frac{\tau^{\frac{1}{2}+\frac{\epsilon}{4}}}{y^{\frac{1}{4}}}.
\end{equation}
Since $\Psi_{\LL}$ constrains $\tau^{\frac{1}{2}-\epsilon} \ll t \ll \tau^{\frac{1}{2}+\epsilon}$, this integral vanishes unless $y$ is larger than a small power of $\tau$.

 We first eliminate the case $\lambda > \tau^{1+\epsilon}$.  If $P_{\pm, \pm} < \frac{1}{100}$ then the derivative of $\sqrt{P_{\pm,\pm}}c(\theta)$ does not vanish and all of the variables are $O(1)$.  Integrating by parts along some line in the $\theta, u, t$ hyperplane saves an arbitrary power of $\lambda$, so that this part of the integral satisfies the required bound.  If instead $P_{\pm, \pm} > \frac{1}{100}$ then either we can integrate by parts with respect to $\theta$ many times to save an arbitrary power of $\lambda$, or else $\theta$ is near its point of stationary phase, in which case $\sqrt{P_{\pm,\pm}}c(\theta) $ is bounded away from 0.  In this case, integrating by parts in $\lambda$ several times gives the required estimate, as the derivative of the phase is bounded away from 0 by at least a constant, see Lemma \ref{phase_fn_lemma}.
 
 In the remaining part of the $\lambda$ integral, restrict to $\lambda \gg_A \tau^{-A}$ by using the weight function.  If $\langle k_{\theta_2} \cdot f_{\pm}, a_{t_2}^{-1}n_{u_2}^{-1} \cdot \tilde{f}_{\pm}\rangle >0$, set $\lambda_0 = \frac{4(\tau - \IM(w))}{2\pi \langle k_{\theta_2}\cdot f_{\pm}, a_{t_2}^{-1}n_{u_2}^{-1}\cdot \tilde{f}_{\pm} \rangle}$.  Outside the interval $\left[\frac{\lambda_0}{2}, 2\lambda_0\right]$, perform a smooth partition of unity into dyadic intervals and integrate by parts several times on each interval.  This obtains the desired bound, since the derivative is bounded below on each interval.
 
 Since $\lambda_0$ satisfies $\partial_\lambda \Phi(\lambda_0) = 0$,  
\begin{equation}
 4 (\tau -\IM(w)) = 2\pi \lambda_0 \sqrt{P_{\pm,\pm}(t_2,u_2)}c(\theta).
\end{equation}
The higher derivatives in $\lambda$ satisfy $\partial^{(n)}_\lambda \Phi(\lambda) = c_n \frac{\tau - \IM(w)}{\lambda^n}.$  We make a smooth partition of unity, expanding $\Phi(\lambda)$ about $\lambda_0$. 
The integral of interest is
\begin{equation}
 \int_{-\infty}^\infty e^{i\Phi(\lambda_0+\lambda)} \frac{\sigma\left(\frac{\pm\lambda}{2^a}\right)}{(\lambda_0 + \lambda)^{\frac{1}{2}+\RE(w)}} d\lambda,
\end{equation}
since the exponential part of the weight is independent of $\lambda$.

Write \begin{equation}\partial_\lambda\Phi(\lambda + \lambda_0) = \frac{4(\tau-\IM(w))\lambda}{\lambda_0 (\lambda+\lambda_0)}\end{equation} and \begin{equation}\tilde{W}(\lambda) = (\lambda_0 + \lambda)^{1 + 4\RE(w)}\sigma\left(\frac{\pm \lambda}{2^a} \right).                                                                                                                                                                                                                                                            \end{equation}
 We have 
 \begin{equation}\left(\frac{\partial}{\partial \lambda}\right)^j \tilde{W}(\lambda) \ll_j \min(2^a, \lambda_0 + \lambda)^{-j}(\lambda_0 + \lambda)^{1 + 4\RE(w)} \|\sigma\|_{C^j},\end{equation} and $\left(\frac{\partial}{\partial \lambda}\right)^j \frac{1}{\partial_\lambda \Phi(\lambda+\lambda_0)} \ll_j \min(2^a, \lambda_0+\lambda)^{-j}\frac{1}{\partial_\lambda \Phi(\lambda+\lambda_0)}$.  Thus with $D_\lambda f = \frac{\partial_\lambda f}{ i \partial_\lambda \Phi(\lambda)}$, we have
\begin{equation}
 (D_\lambda^t)^N \tilde{W}(\lambda) \ll_N \frac{\lambda_0^N(\lambda+\lambda_0)^N}{\tau^N \lambda^N} \frac{1}{\min(2^\alpha, \lambda_0 + \lambda)^N} (\lambda_0 + \lambda)^{1 + 4\RE(w)}\|\sigma\|_{C^N}.
\end{equation}
It follows by integration by parts that the integrals with $2^a > \frac{\lambda_0}{\tau^{\frac{1}{2}-\epsilon}}$ are bounded by $O_A(\tau^{-A})$.

Having restricted to $|\lambda - \lambda_0| \ll \frac{\lambda_0}{\tau^{\frac{1}{2}-\epsilon}}$, it follows that $\lambda \sqrt{P_{\pm,\pm}} \asymp \tau$.
The integral in $\theta$
 \begin{equation}
 \int_0^{1} e^{i 4(\tau - \IM(w)) \log \lambda +2\pi i \lambda \sqrt{P_{\pm, \pm}(t, u)} c(\theta)}d\theta
 \end{equation}
 thus decays rapidly for $|\theta - \frac{1}{2}\zed| \gg \tau^{-\frac{1}{2}+\epsilon}$, so that this part can be discarded.

 Make the change of variable $u := u - \left(\frac{t_1}{t} \right)^2 u_1$.
 The integral in $t$ is 
 \begin{equation}
  \int_0^\infty \exp\left(-\left(\frac{t}{t_1} \right)^2(1 + u^2) - \left(\frac{t_1}{t} \right)^2  \right)e^{2\pi i \lambda \sqrt{P_{\pm,\pm}\left(t, u+ \left(\frac{t_1}{t} \right)^2 u_1\right) }c(\theta)} \frac{dt}{t^3}.
 \end{equation}
 Truncate to $\frac{t}{t_1}, \frac{t_1}{t} \ll \log \tau, u \ll (\log \tau)^2$, making error $O_A(\tau^{-A})$.
 If $t \ll 1$ is sufficiently small, $\partial_t \Phi$ is bounded below, so that integrating by parts several times in $t$ or $u$ shows that the integral is negligible.  
 
 Now assume $t \gg 1$ and let $t_0$ be a point of stationary phase in the phase function.  This implies $\left(u + \left(\frac{t_1}{t_0}\right) u_1\right)^2 \asymp t_0^4$.  Expand about $t_0$ performing a partition of unity.  At $t_0$ we have the partial derivatives,
\begin{equation}
 \partial_{t}^n \sqrt{P_{\pm,\pm}\left(t, u+ \left(\frac{t_1}{t} \right)^2u_1 \right)}\Bigg|_{t=t_0} \asymp_n \frac{1}{t_0^{n-1}},
\end{equation}
see Lemma \ref{phase_fn_lemma}.
Also, $\lambda \sqrt{P_{\pm, \pm}} \asymp \tau$, so $t_0 \asymp \frac{\tau}{\lambda}$.

We have, with $W^*(t) = \frac{\sigma\left(\frac{\pm t}{2^a} \right)}{(t_0+t)^3} \exp\left(-\left(\frac{t_0+t}{t_1} \right)^2(1 + u^2) - \left(\frac{t_1}{t_0+t} \right)^2 \right)$,
\begin{align}
 \partial_{t}^N W^*(t) &\ll_N (\log \tau)^{4N} \min\left(2^a, t_0 + t \right)^{-N} \frac{\|\sigma\|_{C^N}}{(t_0 + t)^3} \\ \notag&\times\exp\left(-\left(\frac{t_0+t}{t_1} \right)^2(1 + u^2) - \left(\frac{t_1}{t_0+t} \right)^2 \right).
\end{align}
Using the estimate for the second derivative, it follows that the stationary phase integral in $t$ at $t_0$ decays rapidly for $|t - t_0| \gg \sqrt{\frac{t_0}{\lambda}}\tau^\epsilon \gg \frac{t_0}{\tau^{\frac{1}{2}-\epsilon}}$. 

For $t \gg 1$, we have, by the constraint  
\begin{equation}
 \partial_{u} \left(\lambda \sqrt{P_{\pm,\pm}\left(t, u + \left(\frac{t_1}{t}\right)^2u_1 \right)} c(\theta) \right) \asymp \frac{\lambda}{\sqrt{P_{\pm,\pm}}} \asymp \frac{\lambda^2}{\tau}.
\end{equation}
The second derivative in $u$ is of order $\frac{\lambda^4}{\tau^3}$.
Thus if $1\ll t \ll \tau^{\frac{1}{2}-\epsilon}$ the integral in $u$ is $\ll_A \tau^{-A}$.  

The above estimates cover all but $E_{\LL}$ which we now estimate.  Subject to the constraint $t_1=\tau^{\frac{1}{2} + O(\epsilon)}$ and $t_0 = t_0(u_1, t_1)$ such that $\frac{t_0}{t_1}, \frac{t_1}{t_0} \ll \log \tau$, which requires $\frac{u_1^2}{t_1^4}, \frac{t_1^4}{u_1^2} = (\log \tau)^{O(1)}$, the weight function in the integral is bounded by
\begin{equation}
 W \ll \frac{\lambda^{1 + 4 \RE(w)}}{t^3} \ll \frac{\tau^{1 + 4 \RE(w)}}{t^{4 + 4 \RE(w)}}
\end{equation}
while the total length of the integral in all variables is \begin{equation}\ll \frac{\lambda_0}{\tau^{\frac{1}{2}-\epsilon}} \tau^{-\frac{1}{2} + \epsilon} \frac{t_0}{\tau^{\frac{1}{2}-\epsilon}} \tau^\epsilon \ll \tau^{-\frac{1}{2} + O(\epsilon)}.\end{equation}
Thus $E_{\LL}\ll \tau^{2 \RE(w) -\frac{3}{2} + O(\epsilon)}$.

\end{proof}

We record the following expression for $E_{\PL}(g_1,w)$.
\begin{lemma}\label{Psi_PL_lemma}
 We have, in $0<\RE(1-s+w)<\frac{3}{8}$,
 \begin{align}
  E_{\PL}(g_1,w) =& \frac{\pi^{-4(1-s+w)}\Gamma(2(1-s+w))}{4(1-s+w)\Gamma(-2(1-s+w))}\\\notag&\times \int_0^\infty \int_{-\infty}^\infty \exp\left(-\left(\frac{t_1}{t} \right)^2 - \left(\frac{t}{t_1} \right)^2 - \left(\frac{t}{t_1}u - \frac{t_1}{t}u_1 \right)^2 \right)\\&\notag\times \psi\left(\frac{\tau^{2+\epsilon}}{y P_{\pm,\pm}(t,u)^2} \right) P_{\pm,\pm}(t,u)^{-2(1-s+w)} du \frac{dt}{t^3}.
 \end{align}

\end{lemma}
\begin{proof}
 We have
 \begin{align}
  E_{\PL}(g_1,w) &= \int_0^\infty\int_{G^1} \lambda^{4(1-s+w)} \exp\left(-\tr g_2^t (g_1^{-1})^t g_1^{-1}g_2 \right)\\
  &\times \notag \psi\left(\frac{\tau^{2+\epsilon}}{y P_{\pm,\pm}(t,u)^2} \right)e\left(\lambda \sqrt{P_{\pm,\pm}(t,u)}c(\theta) \right)\frac{d\lambda}{\lambda} dg_1.
 \end{align}
Integrate in $\theta$ to obtain a $J_0$ Bessel function, then integrate in $\lambda$ to obtain its Mellin transform, which proves the lemma.
\end{proof}

\section{The approximate functional equation in the twisted case}\label{twisted_afe_section}

The goal of this Section is to prove Theorem \ref{twisted_afe_theorem}.

Choose a test function $F_{\pm}(g \cdot f_{\pm}) = n_{\pm} \lambda^{2k}\exp\left(-\lambda^2-t^2 - \frac{1}{t^2} - \left(\frac{u}{t}\right)^2\right)$ with $k \geq 6$ sufficiently large.  Notice that, for $x \in S$ in the singular set, to approach $x$ in homogeneous coordinates from $V_{\pm}$ we must have $\lambda \downarrow 0$ and hence at least one of $u$, $t$ or $\frac{1}{t}$ becomes unbounded, and hence $F_\pm$ is $C^\infty$ on $V_{\bR}$ due to the exponential decay. 

The orbital zeta function is the function
\begin{equation}
\Lambda^{\pm}(s, \phi, F_{\pm}) = \int_{G^+/\Gamma} \phi(g^{-1})(\det g)^{2s}\sum_{x \in V_\zed \cap V_{\pm}} F_{\pm}(g\cdot x) dg.
\end{equation}
The advantage of the test function is that on the time side, integration against the Maass form can be written as a convolution equation that has the Maass form as an eigenfunction.  We don't currently know of a test function which satisfies this property on the Fourier side as well, so that the convolution equation on the Fourier side will be analyzed directly.
\begin{lemma}
 We have the factorization in $\RE(s)>1$,
 \begin{equation}
  \Lambda^{\pm}(s,\phi,F_{\pm}) = \frac{\sqrt{\pi}}{2}K_\nu(2) \Gamma(2s+k)\sL^{\pm}(s,\phi).
 \end{equation}

\end{lemma}
\begin{proof}
 We have
 \begin{align}
  \Lambda^{\pm}(s,\phi,F_{\pm})&= \int_{G^+} \phi(g^{-1}) (\det g)^{2s} \sum_{x \in \Gamma \backslash V_\zed \cap V_{\pm}} \frac{F_{\pm}(g\cdot x)}{|\Stab(x)|}dg\\ \notag
  &=\int_{G^1}\int_0^\infty \phi(g^{-1}) \lambda^{4s} \sum_{x \in \Gamma\backslash V_\zed \cap V_{\pm}} \frac{F_{\pm}(\lambda g \cdot x)}{|\Stab(x)|} \frac{d\lambda}{\lambda} dg\\ \notag
  &= C\int_0^\infty \lambda^{4s+2k}  \exp\left(-\lambda^2 \right) \frac{d\lambda}{\lambda} \cdot \sL^{\pm}(s, \phi) 
 \end{align}
 where $C$ is the eigenvalue of the convolution equation
 $\int_{G_1} \exp(-\tr g^t g) \phi(hg^{-1} )dg = \sqrt{\pi}K_\nu(2) \phi(h)$.  The integral in $\lambda$ evaluates to $\frac{1}{2}\Gamma(2s+k)$.

\end{proof}

By the Poisson summation formula,
\begin{equation}
 \sum_{x \in V_\zed \cap V_{\pm}} F_{\pm}(g \cdot x) = \frac{1}{(\det g)^2} \sum_{\xi \in \hat{V}_\zed} \hat{F}_{\pm}\left(\frac{g \cdot \xi}{\det g}\right).
\end{equation}
This permits the representation
\begin{equation}
 \Lambda^{\pm}(s,\phi,F_{\pm}) = \int_{G^1/\Gamma} \int_0^\infty \lambda^{4(1-s)} \phi(g^{-1}) \sum_{\xi \in \hat{V}_\zed} \hat{F}_{\pm}\left(\lambda g \cdot \xi\right) \frac{d\lambda}{\lambda} dg.
\end{equation}
Let $S = \{x \in V: \Disc(x) = 0\}$ denote the singular points.  Define
\begin{align}
 N^{\pm}(s,\phi,F_{\pm})&=\int_{G^1/\Gamma} \int_0^\infty \lambda^{4(1-s)} \phi(g^{-1}) \sum_{\xi \in \hat{V}_\zed\setminus S} \hat{F}_{\pm}\left(\lambda g \cdot \xi\right) \frac{d\lambda}{\lambda} dg\\ \notag &= \int_{G^1} \int_0^\infty \lambda^{4(1-s)} \phi(g^{-1}) \sum_{\xi \in \Gamma \backslash(\hat{V}_\zed\setminus S)} \frac{\hat{F}_{\pm}\left(\lambda g \cdot \xi\right)}{|\Stab(\xi)|} \frac{d\lambda}{\lambda} dg,\\ \notag
 \Sigma^{\pm}(s,\phi,F_{\pm}) &=\int_{G^1/\Gamma} \int_0^\infty \lambda^{4(1-s)} \phi(g^{-1}) \sum_{\xi \in \hat{V}_\zed \cap S} \hat{F}_{\pm}\left(\lambda g \cdot \xi\right) \frac{d\lambda}{\lambda} dg.
\end{align}
Given $\xi \in \hat{V}_\zed$, let $g_\xi \in G^1$ be such that $g_\xi \cdot f_{\pm} = \frac{\xi}{|\Disc(\xi)|^{\frac{1}{4}}} =\xi_0$.

Let $G(u) = e^{u^2}$. Define 
\begin{align}
 V(g_x,y) =& \frac{2}{\sqrt{\pi}}\frac{\phi(g_x)}{K_\nu(2)}\frac{1}{2\pi i} \oint_{\RE(w)=2} \frac{1}{y^w}\frac{G(w)}{w}\frac{\Gamma(2s+k+2w)}{\Gamma(2s+k)}  dw,\\
 \notag \hat{V}(g_\xi, y)=& \frac{1}{\sqrt{\pi}K_\nu(2)} \frac{1}{2\pi i} \oint_{\RE(w)=2} \frac{1}{y^w} \frac{G(w)}{w} \frac{1}{\Gamma(2s+k)}\\
 \notag &\times \int_0^\infty \lambda^{4(1-s+w)} \int_{G^1} \phi(g^{-1}) \hat{F}_{\pm}(\lambda g g_\xi \cdot f_{\pm}) dg \frac{d\lambda}{\lambda}.
 \end{align}

\begin{proposition}
 We have, with $s = \frac{1}{2} + i\tau$,
 \begin{align}
  \sL(s, \phi) =& \sum_{x \in V_\zed^\pm} \frac{1}{|\Stab(x)| |\Disc(x)|^s}V(g_x, |\Disc(x)|)\\
  \notag &+ \sum_{\xi \in \hat{V}_\zed \setminus S} \frac{1}{|\Stab(\xi)| |\Disc(\xi)|^{1-s}} \hat{V}(g_\xi, |\Disc(\xi)|)\\
  \notag &+ \frac{1}{2\pi i}\oint_{\RE(w) = 2} \Sigma^{\pm}(s-w,\phi, F_{\pm}) G(w) \frac{dw}{w}.
 \end{align}

\end{proposition}
\begin{proof}
 This is the result of picking up the residue at $w = 0$ of $\Lambda^{\pm}(s+w,\phi, F_{\pm})\frac{G(w)}{w}$ as 
 \begin{align}
  \frac{1}{2\pi i} \oint_{\RE(w)=2} \Lambda^{\pm}(s+w, \phi, F_{\pm}) G(w) \frac{dw}{w}- \frac{1}{2\pi i} \oint_{\RE(w)=-2} \Lambda^{\pm}(s+w, \phi, F_{\pm})G(w) \frac{dw}{w}
 \end{align}
and applying Poisson summation in the later integral, then integrating against the resulting series term-by-term.
\end{proof}

\begin{lemma}\label{N_pm_representation}
 We have the representation in $-\frac{1}{2}<\RE(w)$,
 \begin{align}
  N^{\pm}(s-w,\phi, F_{\pm}) &=\frac{\Gamma(2s-2w+k)}{2} \sum_{\xi \in \Gamma \backslash (\hat{V}_\zed\setminus S)} \frac{1}{|\Stab(\xi)||\Disc(\xi)|^{1-s+w}}\\ \notag&\times \int_{G^1}\int_0^\infty \int_{G^1}\lambda^{4(1-s+w)}\phi(g_\xi g_1) \exp\left(-\tr  g_2^t(g_1^{-1})^tg_1^{-1}g_2\right)\\ \notag&\times e(-\lambda \langle  f_{\pm}, g_2^{-1} \cdot \tilde{f}_{\pm}\rangle )  dg_2 \frac{d\lambda}{\lambda} dg_1.
 \end{align}

\end{lemma}

\begin{proof}
We have
\begin{equation}
 N^{\pm}(s-w,\phi, F_{\pm}) = \int_{G^1/\Gamma} \int_0^\infty \lambda^{4(1-s+w)} \phi(g^{-1}) \sum_{\xi \in \hat{V}_\zed \setminus S} \hat{F}_{\pm}(\lambda g \cdot \xi) \frac{d\lambda}{\lambda} dg.
\end{equation}
Write $\xi = |\Disc(\xi)|^{\frac{1}{4}} g_\xi \cdot \tilde{f}_{\pm}$.  Since $F_{\pm}$ is $C^\infty$,  the convergence is absolute due to the decay of $\hat{F}_{\pm}$ when $\RE(w)$ is sufficiently large.  Exchanging sum and integrals and making a change of variable in $\lambda$ obtains
\begin{align}
 &N^{\pm}(s-w,\phi, F_{\pm}) = \\&\notag \sum_{\xi \in \Gamma \backslash (\hat{V}_\zed \setminus S)} \frac{1}{|\Stab(\xi)||\Disc(\xi)|^{1-s+w}}\int_{G^1}\int_0^\infty \lambda^{4(1-s+w)}\phi(g^{-1}) \hat{F}_{\pm}(\lambda g g_\xi \cdot \tilde{f}_{\pm}) \frac{d\lambda}{\lambda} dg. 
\end{align}
Write
\begin{align}
 &\hat{F}_{\pm}(\lambda g g_\xi \cdot \tilde{f}_{\pm}) = \int_{V_{\bR}}F(x)e\left(-\langle x, \lambda g g_\xi \cdot \tilde{f}_{\pm}\rangle \right)dx\\
 \notag &= \int_0^\infty \int_{G^1} \lambda_2^{2k+4} \exp\left(-\lambda_2^2 - \tr g_2^t g_2\right) e\left(-\langle \lambda_2 g_2 \cdot f_{\pm}, \lambda g g_\xi \cdot \tilde{f}_{\pm}\rangle \right) \frac{d\lambda_2}{\lambda_2} dg_2\\ 
 &\notag= \lambda^{-2k-4}\int_0^\infty \int_{G^1} \lambda_2^{2k+4} \exp\left(-\frac{\lambda_2^2}{\lambda^2} -\tr g_2^t g_2\right) e\left(-\lambda_2 \langle g_2 \cdot f_{\pm}, gg_\xi \cdot \tilde{f}_{\pm}\rangle \right) \frac{d\lambda_2}{\lambda_2} dg_2.
\end{align}
Notice that, outside a finite set of points of stationary phase in $g_2$, the integral in $g_2$ is an oscillatory integral and thus saves an arbitrary power of $\lambda_2$, which makes the integral in $\lambda_2$ absolutely convergent.  Making a change of variable $\lambda \mapsto \frac{1}{\lambda}$ and in $g$, and opening the Fourier transform, obtains

\begin{align}
 &N^{\pm}(s-w,\phi,F_{\pm})\\ \notag &= \sum_{\xi \in \Gamma \backslash (\hat{V}_\zed\setminus S)} \frac{1}{|\Stab(\xi)||\Disc(\xi)|^{1-s+w}}\\ \notag&\times \int_{G^1}\int_0^\infty \int_0^\infty\int_{G^1} \lambda_1^{4s+2k-4w}\phi(g_\xi g_1) \lambda_2^{2k+4}\exp\left(-\lambda_1^2\lambda_2^2  - \tr g_2^t g_2 \right)\\& \notag\times e(-\lambda_2 \langle  f_{\pm}, g_2^{-1}g_1^{-1}  \cdot \tilde{f}_{\pm}\rangle) dg_2 \frac{d\lambda_2}{\lambda_2} \frac{d\lambda_1}{\lambda_1}dg_1  \\
 \notag &= \sum_{\xi \in \Gamma \backslash (\hat{V}_\zed\setminus S)} \frac{1}{|\Stab(\xi)||\Disc(\xi)|^{1-s+w}}\\ \notag&\times \int_{G^1}\int_0^\infty\int_0^\infty\int_{G^1} \lambda_1^{4s+2k-4w}\phi(g_\xi g_1) \lambda_2^{2k+4}\exp\left(-\lambda_1^2\lambda_2^2  - \tr g_2^t(g_1^{-1})^t g_1^{-1} g_2 \right)\\& \notag\times e(-\lambda_2 \langle  f_{\pm}, g_2^{-1}  \cdot \tilde{f}_{\pm}\rangle) dg_2 \frac{d\lambda_2}{\lambda_2}\frac{d\lambda_1}{\lambda_1} dg_1 
 \end{align}
 The contribution of the integral at points of stationary phase in $g_2$ is, as a function in $\lambda_2$, possible to expand in falling powers of $\lambda_2$ starting from  $\frac{1}{\lambda_2^{\frac{3}{2}}}$, multiplied by a linear phase in $\lambda_2$.  Taking enough terms is sufficient to gain absolute convergence in the integrals over $\lambda_1$ and $\lambda_2$.  The integral with a linear phase may be treated as a Gaussian integral.  At points of non-stationary phase, the $\lambda_1$ and $\lambda_2$ integrals may be taken together. Alternatively, the integral may be extended from the domain of absolute convergence in $w$ by analytic continuation, which obtains the same value. Performing the integral in $\lambda_1$, which is a Gamma function, this obtains
 \begin{align}
  N^{\pm}(s-w,\phi, F_{\pm}) =& \frac{\Gamma(2s+k-2w)}{2}\sum_{\xi \in \Gamma \backslash (\hat{V}_\zed\setminus S)} \frac{1}{|\Stab(\xi)||\Disc(\xi)|^{1-s+w}}\\ \notag&\times \int_{G^1}\int_0^\infty\int_{G^1} \phi(g_\xi g_1) \lambda^{4(1-s+w)}\exp\left(  - \tr g_2^t(g_1^{-1})^t g_1^{-1} g_2 \right)\\& \notag\times e(-\lambda \langle  f_{\pm}, g_2^{-1}  \cdot \tilde{f}_{\pm}\rangle) dg_2 \frac{d\lambda}{\lambda} dg_1 
 \end{align}
 with the understanding that at the points of stationary phase of $g_2$, the integral in $\lambda$ is only conditionally convergent.
 
\end{proof}

It remains to prove the properties of $V$ and $\hat{V}$, and to estimate the integral against $\Sigma^{\pm}$.
 Using the Lemma \ref{N_pm_representation}, and the definition of the oscillatory integral $E(g_1,w)$  we may express
 \begin{align}
  &\hat{V}(g_\xi,y) = \frac{1}{\sqrt{\pi}K_\nu(2)}\frac{1}{2\pi i} \oint_{\RE(w)=2} \frac{1}{y^w} \frac{G(w)}{w}\frac{\Gamma(2s+k-2w)}{\Gamma(2s+k)} \\ \notag&\times \int_{G^1}\int_0^\infty\int_{G^1}  \lambda^{4(1-s+w)}\phi(g_\xi g_1)\exp\left(-\tr g_2^t(g_1^{-1})^t g_1^{-1}g_2\right) e\left(-\lambda \langle f_{\pm}, g_2^{-1}\cdot\tilde{f}_{\pm}\rangle \right) dg_2 \frac{d\lambda}{\lambda}dg_1dw\\ \notag
&= \frac{1}{\sqrt{\pi}K_\nu(2)} \frac{1}{2\pi i} \oint_{\RE(w)=2} \frac{1}{y^w} \frac{G(w)}{w} \frac{\Gamma(2s+k-2w)}{\Gamma(2s+k)} \int_{G^1} \phi(g_\xi g_1) E(g_1,w) dg_1 dw.
  \end{align}

Recall that we split $E$ as $E_{\SP} + E_{\LL} + E_{\EE} + E_{\PL}$.

\begin{lemma}\label{truncate_w}
 The contribution to $\hat{V}(g_{\xi}, y)$ from $w$ with $|\IM(w)| > \tau^{\frac{1}{2}+\epsilon}$ is $O_A(y^{-2}\tau^{-A})$.
\end{lemma}
\begin{proof}
 We have $F$ is $C^\infty$, and hence \begin{equation}\int_0^\infty \lambda^{4(1-s+w)} \int_{G^1} \phi(g^{-1}) \hat{F}_{\pm}(\lambda g g_\xi \cdot f_{\pm})dg \frac{d\lambda}{\lambda}\end{equation} is bounded uniformly in $\xi$ and $w$.  The claim now follows from the decay in vertical strips of $G(w)$.
\end{proof}

\begin{lemma}\label{truncate_P_pm_pm}
Let $\epsilon>0$.
 The contributition to $\hat{V}(g_\xi, y)$ from $|\IM(w)| \leq \tau^{\frac{1}{2}+\epsilon}$  in which $yP_{\pm,\pm}( t_2, u_2)^2 > \tau^{2+\epsilon}$ is $O_{\epsilon, A}(y^{-2}\tau^{-A})$. 
\end{lemma}
\begin{proof}
Write this contribution as
\begin{align}
 &\frac{1}{2\pi i} \oint_{\substack{\RE(w)=2\\ |\IM(w)| \leq \tau^{\frac{1}{2}+\epsilon}}} \frac{dw}{y^w} \frac{G(w)}{w} \int_0^\infty \frac{d\lambda_1}{\lambda_1}\lambda_1^{4(1-s+w)}\int_{G^1}dg_1 \phi(g_\xi g_1)\int_0^\infty \frac{d\lambda_2}{\lambda_2} \int_{G^1}dg_2\\
 \notag &\times \left[ \lambda_2^{2k+4}\exp(-\lambda_2^2 - \tr g_1^t (g_2^{-1})^t g_2^{-1}g_1) e\left(-\lambda_1\lambda_2 \sqrt{P_{\pm,\pm}(t_2,u_2)}c(\theta)\right) \psi\left(\frac{\tau^{2+\epsilon}}{P_{\pm,\pm}(t_2,u_2)} \right) \right] .
\end{align}
The inner two integrals, against $\lambda_2$ and $g_2$, can be expressed as the Fourier transform of a $C^{\infty}$ function, and thus the inner four integrals, against $\lambda_1, g_1, \lambda_2, g_2$ converge absolutely for $0 < \RE(1-s+w)$.  The resulting function is thus analytic in $w$.  Assume $0 < \RE(1-s+w) < \frac{3}{8}$.  Then the evaluation of Lemma \ref{Psi_PL_lemma} can be used to evaluate $E_{\PL}(g_1,w)$.  Note that this evaluation has analytic continuation to $\RE(1-s+w)>0$.  This obtains
\begin{align}
 &\frac{1}{2\pi i} \oint_{\substack{\RE(w) = 2\\ |\IM(w)| \leq \tau^{\frac{1}{2}+\epsilon}}} \frac{dw}{y^w} \frac{G(w)}{w} \frac{\pi^{-4(1-s+w)}\Gamma(2s+k-2w)\Gamma(2(1-s+w))}{4(1-s+w)\Gamma(2s+k)\Gamma(-2(1-s+w))}\int_{G^1}dg_1 \phi(g_\xi g_1)\\\notag &\int_0^\infty \int_{-\infty}^\infty
  \exp\left(-\left(\frac{t_1}{t_2}\right)^2 - \left(\frac{t_2}{t_1}\right)^2 -\left(\frac{t_2}{t_1}u_2 - \frac{t_1}{t_2}u_1 \right)^2 \right)\\&\notag \psi\left(\frac{\tau^{2+\epsilon}}{y P_{\pm,\pm}(t_2,u_2)^2} \right)P_{\pm,\pm}(t_2,u_2)^{-2(1-s+w)} du_2 \frac{dt_2}{t_2^3}.
\end{align}
To prove the claim, push the contour an arbitrarily large constant to the right.  This may pass poles of $\Gamma(2s+k-2w)$, but each of these satisfies the required bound due to the doubly exponential decay in vertical strips of $G(w)$.  This also suffices to bound the horizontal integrals that result.  The claim now follows on bounding the integral in absolute value, putting in a sup-bound for the Maass form.

\end{proof}

\begin{lemma}\label{V_function_lemma}
 The functions $V$ and $\hat{V}$ satisfy the following properties.

\begin{enumerate}
 \item There is a $c>0$ such that, if $g = n_u a_t k_\theta$ then, for any $\epsilon>0$, \begin{equation}V(g,y), \hat{V}(g,y) \ll \exp(-ct^c) + O_\epsilon\left(y^{-2} \tau^{-\frac{3}{2} + \epsilon}\right).\end{equation}
 \item For any fixed $\eta > 0$, if $h = n_{u'}a_{t'}k_{\theta'}$ with $|u'|, |1-t'|, |\theta'|<\delta < \tau^{-\eta}$ and $t \ll (\log \tau)^{O(1)}$ then, for all $\epsilon >0$ \begin{equation}V(hg,y)-V(g,y), \hat{V}(hg,y)-\hat{V}(g,y) \ll_{\eta, \epsilon} \delta (1 + t)^{O(1)}+y^{-2}\tau^{-\frac{3}{2}+\epsilon}.\end{equation}
 \item   For $a \geq 0$,
 \begin{equation}
  y^a \left(\frac{\partial}{\partial y} \right)^a V^{(a)}(g,y), y^a \left(\frac{\partial}{\partial y}\right)^a \hat{V}(g,y) \ll_{a,A} \left(1 + \frac{y}{1+|\tau|^2}\right)^{-A} + y^{-2}\tau^{-A}.
 \end{equation}

\end{enumerate}

\end{lemma}
\begin{proof}
We have $V(g,y) = \phi(g) V(y)$ where $V(y)$ can be analyzed as in \cite{IK04} Lemma 5.4.  This suffices to prove the bounds for $y^a \partial_y^a V(g,y)$.  The decay in $t$ follows from the rapid decay of the cusp form. Let $h = n_{u_1}a_{t_1}k_{\theta_1}$, $g = n_{u_2}a_{t_2}k_{\theta_2}$, then
\begin{equation}
 hg = n_{u_1}a_{t_1} n_{u_2}a_{t_2}(n_{u_2}a_{t_2})^{-1}k_{\theta_1}(n_{u_2} a_{t_2})k_{\theta_2}.
\end{equation}
Write
\begin{equation}
 (n_{u_2}a_{t_2})^{-1} k_{\theta_1} n_{u_2}a_{t_2} = n_{u_1'}a_{t_1'}k_{\theta_1'},
\end{equation}
with $|u_1'|, |t_1'-1|, |\theta_1'| = O\left(\delta t_2^{O(1)}\right)$.
Thus 
\begin{equation}
 \phi(hg) - \phi(g) = \phi(n_{u_1 + t_1^2 u_2 + t_1^2t_2^2 u_1'}a_{t_1t_2t_1'}) - \phi(n_{u_2}a_{t_2}) = O\left(\delta t_2^{O(1)}\right).
\end{equation}

Recall that $\hat{V}$ is given by 
\begin{equation}
 \hat{V}(g,y) = \frac{1}{K_\nu(2)\sqrt{\pi}} \frac{1}{2\pi i} \oint_{\RE(w)=2} \frac{1}{y^w} \frac{G(w)}{w} \frac{\Gamma(k+2s-2w)}{\Gamma(k+2s)} \int_{g_1 \in G^1} \phi(gg_1) E(g_1,w)dg_1dw,
\end{equation}
with
\begin{align}
 &E(g_1, w) =\\
  \notag & \int_0^\infty\int_{G^1}  \lambda^{4(1-s+w)} \exp\left(- \tr g_2^t (g_1^{-1})^t g_1^{-1}g_2 \right)e\left(\lambda \sqrt{P_{\pm,\pm}(t,u)}c(\theta)\right) \frac{dt}{t^3} du d\theta \frac{d\lambda}{\lambda}
\end{align}
and the trace in the exponential given by
 \begin{equation}
 \left(\frac{t_1}{t_2}\right)^2 + \left(\frac{t_2}{t_1}\right)^2 + \left(\frac{t_2}{t_1}u_2 - \frac{t_1}{t_2}u_1 \right)^2.
\end{equation}

By Lemmas \ref{truncate_w} and \ref{truncate_P_pm_pm}, the part of the integral in $w$ with $|\IM(w)|> \tau^{\frac{1}{2}+\epsilon}$ and the part of the integral over $g_2$ with $y P_{\pm,\pm}^2 > \tau^{2+\epsilon}$ can be discarded while making error $O_{\epsilon,A}(y^{-2}\tau^{-A})$.  Similarly, arguing as in Lemma \ref{E_est_lemma}, truncate to $\lambda = \tau^{O(1)}$ with the same error.  The above considerations truncate $\lambda, t_2$ and $u_2$ in $E(g_1,w)$ to all have size $\tau^{O(1)}$.  Using the exponential factor in this integral, truncate the Iwasawa coordinates in $g_1$ to $u_1, t_1 = \tau^{O(1)}$ with the same error.  Now apply the bounds in Lemma \ref{E_est_lemma} to reduce $E(g_1,w)$ to the two stationary phase terms $E_{\SP}$ and $E_{\PL}$ while making error $O(y^{-2}\tau^{-A})$, in the case of $E_{\SP}$ we can take $u_1, t_1 = (\log \tau)^{O(1)}$.

When $y < \tau^2$, in the terms from $E_{\SP}$ shift the contour to $\RE(w) = -\epsilon$, bounding the horizontal integrals by $O_A(\tau^{-A})$. There is a pole at 0, which we bound together with the $w$ integral.  Due to the rapid decay of the $w$ integral, and applying the bounds of Lemma \ref{E_est_lemma} these  are bounded by
\begin{equation}
 O_A(\tau^{-A}) + \sum_{(\lambda_0, \theta_0, t_0, u_0) \in \sS}\int_{G^1} \phi(gg_1) \exp\left(-\left(\frac{t_1}{t_0} \right)^2 - \left(\frac{t_0}{t_1} \right)^2 - \left(\frac{t_0}{t_1} u_0 - \frac{t_1}{t_0} u_1 \right)^2 \right)\frac{dt_1}{t_1^3} du_1 d\theta_1.
\end{equation}
Write \begin{align}gg_1 = n_u a_t k_\theta n_{u_1} a_{t_1}k_{\theta_1} = n_u a_t n_{u_1} a_{t_1} ((n_{u_1} a_{t_1})^{-1} k_\theta n_{u_1} a_{t_1}) k_{\theta_1}.\end{align}
Let $n_{u_2}a_{t_2} k_{\theta_2} = ((n_{u_1} a_{t_1})^{-1} k_\theta n_{u_1} a_{t_1})$, so that $t_2 = u_1^{O(1)} t_1^{O(1)}$.  Hence if $g g_1 = n_{u'} a_{t'} k_{\theta'}$ then $t' = t t_1^{O(1)} u_1^{O(1)}$.  Since $\phi$ decays singly exponentially in $t'$ and the exponential factor decays doubly exponentially in $t_1$ and $t_1 u_1$, the integral is bounded by $\exp(-c t^{c})$ for some $c > 0$.
If $y < \tau^\epsilon$, there is an additional contribution from $\E_{\PL}$ which occurs when $\tau^{1-O(\epsilon)} < |u_1| < \tau^{1 + O(\epsilon)}$, and $\frac{u_1^2}{t_1^4} = (\log \tau)^{O(1)}$.  In this case integrate the $w$ integral on the $\RE(w) = 2$ line.  The ratio of Gamma factors is bounded by $\ll \tau^{-2 \RE(w)}$, which balances the $w$ contribution from $E_{\PL}$.  Bounding the Maass form by a constant, this obtains a bound of 
\begin{align}
 O_A(\tau^{-A}) + \int_{\tau^{1-O(\epsilon)}}^{\tau^{1 + O(\epsilon)}} du_1 \int_{|u_1|^{\frac{1}{2}}\tau^{-\epsilon}}^{|u_1|^{\frac{1}{2}} \tau^\epsilon} \frac{dt_1}{t_1^3}  y^{-2} \tau^{-\frac{3}{2} + O(\epsilon)}= O_A(\tau^{-A}) + O\left(y^{-2} \tau^{-\frac{3}{2} + O(\epsilon)}\right).
\end{align}
When $y > \tau^2$, only $E_{\SP}$ needs to be considered.  To obtain the estimate in (1), bound the integral between the Maass form and exponential function as before with $\RE(w)=2$.

To prove (2), in both $\hat{V}(hg, y)$ and $\hat{V}(g,y)$, bound all but the main term from $E_{\SP}$ as for (1), absorbing the errors in the second error term.  Arguing as before, this reduces to bounding the integral
\begin{align}
 &\sum_{(\lambda_0, \theta_0, t_0, u_0) \in \sS} \int_{u_1, t_1 = (\log t)^{O(1)}}\left[\phi(hgg_1)-\phi(gg_1)\right]\\ &\notag \times \exp\left(-\left(\frac{t_1}{t_0} \right)^2 - \left(\frac{t_0}{t_1} \right)^2 - \left(\frac{t_0}{t_1} u_0 - \frac{t_1}{t_0} u_1 \right)^2 \right) \frac{dt_1}{t_1^3} du_1 \ll \delta (1+ t)^{O(1)}.
\end{align}
As above, write $gg_1 = n_{u'} a_{t'} k_{\theta'}$ with $u' = u + u_1^{O(1)}t_1^{O(1)}$, $t' = t u_1^{O(1)}t_1^{O(1)}$.  Then as before, 
\begin{equation}
 |\phi(hgg_1)-\phi(gg_1)| \ll\delta\left(1 + t u_1^{O(1)}t_1^{O(1)}\right)^{O(1)}
\end{equation}
and integrating against the exponential factor obtains the claimed bound.

To prove (3), differentiate under the integral sign.  The proof of the estimate is the same as that for (1) when $y \ll \tau^2$.  When $y \gg \tau^2$, bound $E_{\SP}$ by pushing the integral arbitrarily far to the right and estimate as before.
\end{proof}

\subsection{The singular terms}
In this Section we prove the following estimate for singular forms, which, combined  with Lemma \ref{V_function_lemma}  proves Theorem \ref{twisted_afe_theorem}.
\begin{lemma}\label{singular_form_lemma}
 We have the bound 
 \begin{equation}
\frac{2}{\sqrt{\pi}K_\nu(2)\Gamma(2s+k)}   \frac{1}{2\pi i} \oint_{\RE(w)=2} \Sigma^{\pm}(s-w, \phi, F_{\pm}) G(w) \frac{dw}{w} = O_A(\tau^{-A}).
 \end{equation}

\end{lemma}

Taken together the singular terms have sufficient decay for the integral to be convergent, but taken separately they can diverge. 
We now regularize the singular terms by introducing a smooth weight $w_z(\lambda) = e^{-z(\lambda + \lambda^{-1})}$ which tends to 1 as $z \to 0$.  Consider  
\begin{align}
 \Sigma^{\pm}(s-w, \phi, F_{\pm}) &= \lim_{z\downarrow 0}\int_{G^1/\Gamma} \int_0^\infty w_z(\lambda) \lambda^{4(1-s+w)} \phi(g^{-1}) \sum_{\xi \in \hat{V}_\zed \cap S} \hat{F}_{\pm}(\lambda g \cdot \xi) \frac{d\lambda}{\lambda} dg.
\end{align}
Note that $\hat{F}_\pm$ is $K$ invariant so that the integrals over $K$ are omitted in this section.
The singular terms have the fibration 
\begin{equation}
 \hat{L}_0 = \{0\} \sqcup \bigsqcup_{m=1}^\infty \bigsqcup_{\gamma \in \Gamma/(\Gamma \cap N)}\{\gamma\cdot (0,0,0,m)\} \sqcup \bigsqcup_{m=1}^\infty \bigsqcup_{n=0}^{3m-1} \bigsqcup_{\gamma \in \Gamma} \{\gamma \cdot(0,0,3m,n)\}.
\end{equation}
As in \cite{H19}, integration over the first two terms of the fibration vanishes due to the cusp form.  To see this, in the case of $\{0\}$, the integral is
\begin{equation}
\lim_{z \downarrow 0}\hat{F}_{\pm}(0)\int_0^\infty w_z(\lambda)\lambda^{4(1-s+w)} \frac{d\lambda}{\lambda} \int_{G^1/\Gamma} \phi(g^{-1}) dg = 0
\end{equation}
since $\phi$ is mean 0.  Similarly, the second term in the fibration is given by
\begin{equation}
 \lim_{z \downarrow 0} \int_0^\infty w_z(\lambda) \lambda^{4(1-s+w)}\frac{d\lambda}{\lambda} \int_0^1 du \int_0^\infty \frac{dt}{t^3} \phi(n_{u}^t a_{\frac{1}{t}}) \sum_{m = 1}^\infty \hat{F}_{\pm}(\lambda a_{\frac{1}{t}} \cdot (0,0,0,m)) 
\end{equation}
which again vanishes by integrating in $u$, since $(0,0,0,m)$ is stabilized by $n_u$ and the cusp form does not have a constant term in its Fourier expansion.

The remaining singular terms are given by
\begin{align}
  \Sigma^{\pm}(s-w, \phi, F_{\pm})=&\lim_{z \downarrow 0}\int_0^\infty w_z(\lambda) \lambda^{4(1-s+w)} \int_{-\infty}^\infty \int_{0}^\infty \phi(n_{u}^t a_{\frac{1}{t}}) \\&\notag\times\sum_{m=1}^\infty \sum_{n=0}^{3m-1} \hat{F}(\lambda  a_{\frac{1}{t}} n_u \cdot (0,0,3m,n)) \frac{d\lambda}{\lambda} \frac{dt}{t^3} du.
 \end{align}
 Insert the Fourier series for $\phi$ to obtain
 \begin{align}
 \Sigma^{\pm}(s-w, \phi, F_{\pm})=&
 \lim_{z \downarrow 0}2\int_0^\infty w_z(\lambda) \lambda^{4(1-s+w)}\int_{-\infty}^\infty \int_0^\infty \sum_{\ell = 1}^\infty \rho_\phi(\ell) K_{\nu}(2\pi \ell t^2) \cos(2\pi \ell u)\\ \notag&\times \sum_{m=1}^\infty \sum_{n=0}^{3m-1} \hat{F}\left(0,0, \frac{3m\lambda}{t}, \frac{\lambda(n + 3um)}{t^3}\right) \frac{d\lambda}{\lambda}\frac{dt}{t^2} du.
 \end{align}
 Change variable in $u$ and sum in $n$ to select $3m|\ell$ which we rewrite as $\ell = 3m\ell'$.  This obtains 
 \begin{align}
 \Sigma^{\pm}(s-w, \phi, F_{\pm})=&\lim_{z \downarrow 0} 2\int_0^\infty \lambda^{4(1-s+w)-1} \int_{-\infty}^\infty \int_0^\infty \sum_{\ell, m=1}^\infty\rho_\phi(3\ell m) K_\nu(6\pi \ell m t^2) \\\notag &\times \cos\left(\frac{6\pi \ell m t^3 u}{\lambda}\right) \hat{F}\left(0,0, \frac{3m\lambda}{t}, u\right)\frac{d\lambda}{\lambda} tdt du.
 \end{align}
 Integrating in $u$, and writing $\hat{F}_4$ for the Fourier transform taken in the first three variables but not the fourth, obtains
 \begin{align}
 \Sigma^{\pm}(s-w, \phi, F_{\pm})=&\lim_{z \downarrow 0} \sum_{\varepsilon = \pm }  \int_0^\infty w_z(\lambda)\lambda^{4(1-s+w)-1} \int_0^\infty\\ \notag &\times \sum_{\ell, m=1}^\infty \rho_\phi(3\ell m)K_\nu(6\pi \ell m t^2)\hat{F}_4\left(0,0, \frac{3m\lambda}{t}, \frac{\varepsilon 3\ell m t^3}{\lambda} \right)\frac{d\lambda}{\lambda} tdt.
\end{align}
Evidently there is no longer a issue of convergence so that the weight $w_z(\lambda)$ may be removed.  This is because the exponential decay of $K_\nu$ controls the large $t,m,\ell$ behavior, while when $t$ is small, the third slot of the Fourier transform essentially bounds $\lambda \ll t$ and $\ell m \ll t^{-2}$ which gains the convergence in $t$ for $\RE(w) > -\frac{1}{4}.$  

Let $\alpha = \frac{3m\lambda}{t}$, $\beta = \frac{3\ell m t^3}{\lambda}$.  This obtains
\begin{align}
 \Sigma^{\pm}(s-w, \phi, F_{\pm})=&\frac{1}{9} \sum_{\varepsilon = \pm}\int_0^\infty \int_0^\infty \alpha^{6(1-s+w)-\frac{1}{2}} \beta^{2(1-s+w)-\frac{1}{2}}\\ \notag &\times \sum_{\ell, m=1}^\infty \frac{\rho_\phi(3\ell m) K_\nu\left(\frac{2\pi \alpha \beta}{3m}\right)}{\ell^{2(1-s+w)+\frac{1}{2}}(3m)^{8(1-s+w)}}\hat{F}_4(0,0,\alpha, \varepsilon \beta) \frac{d\alpha}{\alpha}d\beta
 \end{align}
 We now open the Fourier transform as an integral over $V_{\bR}$. This gives
 \begin{align}
 \Sigma^{\pm}(s-w, \phi, F_{\pm})=&\frac{1}{9}\sum_{\varepsilon = \pm}\int_0^\infty \int_0^\infty \alpha^{6(1-s+w)-\frac{1}{2}} \beta^{2(1-s+w)-\frac{1}{2}} \sum_{\ell, m=1}^\infty \frac{\rho_\phi(3\ell m) K_\nu\left(\frac{2\pi \alpha \beta}{3m}\right)}{\ell^{2(1-s+w)+\frac{1}{2}}(3m)^{8(1-s+w)}} \\&\notag\times\int_{\bR^3} F_{\pm}(\varepsilon \beta, x_2, x_3, x_4) e^{-2\pi i \frac{\alpha x_2}{3}} \frac{d\alpha}{\alpha} d\beta dx
  \end{align}
 We now write the integral over $\beta, x_2, x_3, x_4$ in $V_{\bR}$ as an integral over $G^+$ using the orbit description of $F_{\pm}$. In what follows we give the proof for $V_-$, the proof in the case of $V_+$ being similar.  In this case, the first coordinate $\beta$ is given by $\frac{\lambda s(\theta)}{\sqrt{2}t^3}$, which we impose the constraint is positive.  This obtains
    \begin{align}
 \Sigma^{\pm}(s-w, \phi, F_{\pm})&=  \frac{1}{9}\int_0^\infty \alpha^{6(1-s+w)-\frac{1}{2}} \int_{\theta\in [0, 1]} \int_{-\infty}^\infty \int_0^\infty \int_0^\infty \left|\frac{\lambda}{\sqrt{2}} \frac{s( \theta)}{t^3} \right|^{2(1-s+w)-\frac{1}{2}} \\&\times \notag \sum_{\ell, m=1}^\infty \frac{\rho_\phi(3\ell m)K_\nu\left(\frac{2\pi \alpha}{3m}\left(\frac{\lambda s(\theta)}{\sqrt{2}t^3} \right)\right)}{\ell^{2(1-s+w)+\frac{1}{2}}(3m)^{8(1-s+w)}}\\&\notag \times e^{-\frac{2\pi i \alpha}{3}\frac{\lambda}{\sqrt{2}t}\left(c(\theta) +3 s(\theta)u \right)} \lambda^{2k}\exp\left(-\lambda^2  -t^2 - \frac{1}{t^2} - (tu)^2 \right) \lambda^3 d\lambda \frac{dt}{t} du d\theta \frac{d\alpha}{\alpha}
   \end{align}
   After a change of variable in $\alpha$ this obtains
   \begin{align}
  \Sigma^{\pm}(s-w, \phi, F_{\pm})&= \frac{1}{9}\int_0^\infty \alpha^{6(1-s+w)-\frac{1}{2}} \int_{\theta\in [0, 1]} \int_{-\infty}^\infty \int_0^\infty \int_0^\infty \left|\frac{\lambda}{\sqrt{2}} \frac{s( \theta)}{t^3} \right|^{-4(1-s+w)} \\\notag&\times \sum_{\ell, m=1}^\infty \frac{\rho_\phi(3\ell m)K_\nu\left(2\pi \alpha\right)}{(3\ell m)^{2(1-s+w)+\frac{1}{2}}}e^{-2\pi i \alpha mt^2\left(\frac{c(\theta)}{s(\theta)} +3 u \right)} \lambda^{2k}\\&\times \notag\exp\left(-\lambda^2  -t^2 - \frac{1}{t^2} - (tu)^2 \right) \lambda^3 d\lambda \frac{dt}{t} du d\theta \frac{d\alpha}{\alpha}.
  \end{align}
  Integrate in $\theta$ treating this as a cotangent integral, and integrate in $\lambda$.  This obtains 
  \begin{align}
     \Sigma^{\pm}(s-w, \phi, &F_{\pm})=\frac{1}{18}\frac{1}{(2\pi)^{-(1-2s+2w)}} \frac{\Gamma(k+2s-2w)}{\Gamma(-1 + 2s -2w)}\int_0^\infty \alpha^{4(1-s+w)}  \\\notag&\times \int_{-\infty}^\infty \int_0^\infty t^{1 + 8(1-s+w)}\sum_{\ell,m=1}^\infty \frac{\rho_\phi(3\ell m)K_\nu(2\pi \alpha)K_{\frac{1}{2}+(1-2s+2w)}(2\pi \alpha mt^2)}{(3\ell)^{2(1-s+w)+\frac{1}{2}}m^{4(1-s+w)}}\\&\notag\times e^{-6\pi i \alpha mt^2 u} \exp\left(-t^2-\frac{1}{t^2} - (tu)^2\right) \frac{dt}{t} du \frac{d\alpha}{\alpha}.
  \end{align}
Now integrate in $u$
\begin{align}
\Sigma^{\pm}(s-w, \phi, F_{\pm})= &\frac{1}{18\sqrt{\pi}}\frac{1}{(2\pi)^{-(1-2s+2w)}} \frac{\Gamma(k+2s-2w)}{\Gamma(-1 + 2s -2w)}\int_0^\infty \alpha^{4(1-s+w)}\int_0^\infty t^{8(1-s+w)} \\&\notag\times \sum_{\ell,m=1}^\infty \frac{\rho_\phi(3\ell m)K_\nu(2\pi \alpha)K_{\frac{1}{2}+(1-2s+2w)}(2\pi \alpha mt^2)}{(3\ell)^{2(1-s+w)+\frac{1}{2}}m^{4(1-s+w)}}\\&\notag\times  \exp\left(-t^2-\frac{1}{t^2} - 9\pi^2 \alpha^2 m^2 t^2\right) \frac{dt}{t} \frac{d\alpha}{\alpha}.
 \end{align}
 Make a change of variable in $\alpha$ to obtain
 \begin{align}
  \Sigma^{\pm}&(s-w, \phi, F_{\pm})= \frac{1}{18\sqrt{\pi}}\frac{1}{(2\pi)^{-(1-2s+2w)}} \frac{\Gamma(k+2s-2w)}{\Gamma(-1 + 2s -2w)}\int_0^\infty \alpha^{4(1-s+w)}\int_0^\infty  \\\notag&\times \sum_{\ell,m=1}^\infty \frac{\rho_\phi(3\ell m)K_\nu\left(\frac{18\pi \alpha m^2}{t^2}\right)K_{\frac{1}{2}+(1-2s+2w)}(18\pi \alpha m^3) 3^{6(1-s+w)-\frac{1}{2}}m^{4(1-s+w)}} {\ell^{2(1-s+w)+\frac{1}{2}}}\\&\notag\times  \exp\left(-t^2-\frac{1}{t^2} - \frac{81\pi^2 \alpha^2 m^4}{t^2}\right) \frac{dt}{t} \frac{d\alpha}{\alpha}.
 \end{align}
 The integral in $t$ is now one of the standard $K$-Bessel function integrals.  Exercising this obtains
 \begin{align}
& \Sigma^{\pm}(s-w, \phi, F_{\pm})= \frac{K_\nu(2)}{18\sqrt{\pi}}\frac{1}{(2\pi)^{-(1-2s+2w)}} \frac{\Gamma(k+2s-2w)}{\Gamma(-1 + 2s -2w)}\int_0^\infty \alpha^{4(1-s+w)}  \\\notag&\times \sum_{\ell,m=1}^\infty \frac{\rho_\phi(3\ell m)K_\nu\left(9\pi \alpha m^2\right)K_{\frac{1}{2}+(1-2s+2w)}(18\pi \alpha m^3) 3^{6(1-s+w)-\frac{1}{2}}m^{4(1-s+w)}} {\ell^{2(1-s+w)+\frac{1}{2}}} \frac{d\alpha}{\alpha}
\end{align}
After a further change of variable in $\alpha$ this becomes
 \begin{align}
 \Sigma^{\pm}(s-w, \phi, F_{\pm})&= \frac{K_\nu(2)}{18\sqrt{\pi}}\frac{1}{(2\pi)^{-(1-2s+2w)}} \frac{\Gamma(k+2s-2w)}{\Gamma(-1 + 2s -2w)}\int_0^\infty \alpha^{4(1-s+w)}  \\&\notag\times \sum_{\ell,m=1}^\infty \frac{\rho_\phi(3\ell m)K_\nu\left(\frac{\pi \alpha}{m} \right)K_{\frac{1}{2}+(1-2s+2w)}(2\pi \alpha ) } {(3\ell)^{2(1-s+w)+\frac{1}{2}}m^{8(1-s+w)}} \frac{d\alpha}{\alpha}.
\end{align}
The integral in $\alpha$ may now finally be expressed as in Lemma \ref{double_bessel_mellin_eval},
\begin{align}
 \int_0^\infty &K_{\frac{1}{2} + (1-2s+2w)}(2\pi \alpha) K_\nu\left(\frac{\pi \alpha}{m}\right) \alpha^{4(1-s+w)} \frac{d\alpha}{\alpha}\\ \notag
 &= \frac{\pi^{-4(1-s+w)}}{2^{3+\nu}m^\nu}\Gamma\left(\frac{1}{2}\left(6(1-s+w)-\frac{1}{2} + \nu \right) \right)\Gamma\left(\frac{1}{2}\left(2(1-s+w)+\frac{1}{2} - \nu \right) \right)\\
 &\times \notag \int_0^1  t^{1-s+w-\frac{3}{4} - \frac{\nu}{2}}(1-t)^{3(1-s+w)-\frac{5}{4}-\frac{\nu}{2}} \left(1 - \frac{4m^2-1}{4m^2} t\right)^{-3(1-s+w)+\frac{1}{4}-\frac{\nu}{2}}dt.
\end{align}
In the last integral, bound $\left(1 - \frac{4m^2-1}{4m^2} t\right)^{-3(1-s+w)+\frac{1}{4}-\frac{\nu}{2}} \ll m^{6\RE(1-s+w)}$ so that the integral altogether is 
\begin{equation}
 \int_0^1  t^{1-s+w-\frac{3}{4} - \frac{\nu}{2}}(1-t)^{3(1-s+w)-\frac{5}{4}-\frac{\nu}{2}} \left(1 - \frac{4m^2-1}{4m^2} t\right)^{-3(1-s+w)+\frac{1}{4}-\frac{\nu}{2}}dt \ll m^{6\RE(1-s+w)}.
\end{equation}

\begin{proof}[Proof of Lemma \ref{singular_form_lemma}]
 Shift the contour to $\RE(w) = \epsilon$.  This obtains
 \begin{align}
  &\frac{2}{\sqrt{\pi}K_\nu(2)\Gamma(2s+k)}\frac{1}{2\pi i} \oint_{\RE(w) = \epsilon} \Sigma^{\pm}(s-w, \phi, F_{\pm}) G(w)\frac{dw}{w}\\
  \notag &= \frac{1}{9} \oint_{\RE(w) = \epsilon} \frac{G(w)}{w} 2^{-2-2s+2w}\pi^{-4+2s-2w}\\\notag &\times \frac{\Gamma\left(2s+k-2w\right)\Gamma\left(3(1-s+w)-\frac{1}{4}+\frac{\nu}{2} \right)\Gamma\left(1-s+w + \frac{1}{4}-\frac{\nu}{2} \right)}{\Gamma\left(2s+k \right)\Gamma\left(-1+2s-2w \right)}\\& \notag \times \sum_{\ell, m=1}^\infty \frac{\rho_\phi(3\ell m)}{(3\ell)^{2(1-s+w)+\frac{1}{2}}m^{8(1-s+w)}}\\&\notag\times\int_0^1 t^{1-s+w -\frac{3}{4}-\frac{\nu}{2}}(1-t)^{3(1-s+w)-\frac{5}{4}-\frac{\nu}{2}}\left(1-\frac{4m^2-1}{4m^2} \right)^{-3(1-s+w)+\frac{1}{4}-\frac{\nu}{2}}dt dw.
 \end{align}
The sum over $\ell$ and $m$ and the integral in $t$ are uniformly bounded.  The claim now follows since $G(w)$ decays doubly exponentially in $\IM(w)$, so that we can truncate to $|\IM(w)| < \tau^{\frac{1}{2}+\epsilon}$ with error $O_A(\tau^{-A})$.  In this range, by Stirling's approximation the ratio of Gamma factors is $O_A(\tau^{-A})$.
\end{proof}

\bibliographystyle{plain}

\end{document}